\documentclass{article}
\usepackage[dvips]{graphicx}
\usepackage{amsmath}
\usepackage{amsthm}
\usepackage{amssymb}
\usepackage[round]{natbib}
\usepackage{verbatim}
\usepackage{mathtools}
\usepackage{subcaption}
\usepackage{algorithm2e}

\usepackage{hyperref}

\usepackage{xcolor}
\definecolor{dark-red}{rgb}{0.4,0.15,0.15}
\definecolor{dark-blue}{rgb}{0,0,0.7}
\hypersetup{
    colorlinks, linkcolor={dark-blue},
    citecolor={dark-blue}, urlcolor={dark-blue}
}

\mathtoolsset{showonlyrefs}

\newtheorem{theorem}{Theorem}
\newtheorem{conjecture}{Conjecture}[]
\newtheorem{corollary}{Corollary}[]
\newcommand{\N}{n}

\newcommand{\z}{\mathbf{z}}
\newcommand{\x}{\mathbf{x}}
\newcommand{\Z}{\mathbf{Z}}
\newcommand{\X}{\mathbf{X}}
\newcommand{\U}{\mathbf{U}}
\newcommand{\uu}{\mathbf{u}}

\newcommand{\mz}{\mathcal{Z}_\alpha^\mu}
\newcommand{\nz}{\overline{\mathcal{Z}}_\alpha^\mu}
\newcommand{\definedas}{\overset{\underset{\mathrm{def}}{}}{=}}

\begin{document}
    
\title{A New Confidence Interval for the Mean of a Bounded Random Variable}

\author{Erik~Learned-Miller and Philip S. Thomas\\College of Information and Computer Sciences\\University of Massachusetts\\\{elm,pthomas\}@cs.umass.edu}

\date{}

\maketitle

\begin{abstract}
    We present a new method for constructing a confidence interval for the mean of a bounded random variable from samples of the random variable. We conjecture that the confidence interval has guaranteed coverage, i.e., that it contains the mean with high probability for all distributions on a bounded interval, for all samples sizes, and for all confidence levels. This new method provides confidence intervals that are competitive with those produced using Student's $t$-statistic, but does not rely on normality assumptions.  In particular, its only requirement is that the distribution be bounded on a known finite interval.
\end{abstract}
\section{Introduction}

Consider one of the fundamental problems in statistics: how to use $n$ samples of a real-valued random variable to obtain a confidence interval on its mean. 
Methods for constructing such confidence intervals are used across all branches of science. 
In the natural and social sciences, the confidence interval based on Student's $t$-statistic \citep{student1908probable} is one of the standard tools for quantifying uncertainty about the results of an empirical study. 
In theoretical work, concentration inequalities like Hoeffding's inequality \citep{Hoeffding1963} are often used to analyze properties of algorithms in machine learning, data science, and other areas. 
Providing methods for obtaining tighter confidence intervals from fewer samples is critical to scientific advancement, enabling stronger conclusions to be drawn from the same experimental data. 

In particular, there is a practical need today for confidence intervals that hold for small sample sizes. 
Since the confidence interval produced using Student's $t$-statistic, which we refer to hereafter as the \emph{Student-$t$ interval}, relies on the (near) normality of the sample mean, it is recommended that sample sizes be at least 30 for it to be used, unless there is
a specific reason to believe that the population distribution is approximately normal. 
While other confidence intervals that hold for small
sample sizes exist (such as Anderson's~\citeyearpar{Anderson1969}), they produce intervals that are so wide as to be of little use in practice. This
leaves the practitioner with the following choices:
\begin{itemize}
    \item use methods, such as bootstrap methods, with no performance guarantees;
    \item use methods with unrealistic assumptions, such as the Student-$t$ interval;
    \item use valid but weak methods such as Hoeffding or Anderson's inequalities that provide little information about the mean;
    \item abandon the idea of obtaining useful confidence intervals from the data.
\end{itemize}
In this paper, we introduce a new confidence interval for bounded distributions that is much tighter than other confidence intervals
that come with guarantees. We conjecture that it holds for all bounded distributions, all samples sizes, and all confidence levels. 
We suggest that for many applications, this is the first practical confidence interval for sample sizes less than 30. We now offer a formal statement of the problem we are addressing.

We became aware of similar previous work after the initial publication of this paper \citep{gaffke2005three}. We thank Aaditya Ramdas for bringing it to our attention.

\subsection{Problem Statement}
Let $X_1,\dotsc,X_\N$ be $\N$ independent and identically distributed real-valued random variables. 
Let each $X_i$ take values in the interval $[0,1]$ and have expected value $\mu$. 
For now we focus on a high-confidence \emph{upper} bound, i.e., we desire a function $m_\alpha$ such that, for all distributions of $X_i$, all sample sizes $n\geq 1$, and all confidence levels $1-\alpha \in [0,1]$:
\begin{equation}
    \label{eq:mainResult}
    \Pr\left ( m_\alpha(X_1,\dotsc,X_\N) \geq \mu \right ) \geq 1-\alpha.
\end{equation} 
That is, with probability at least $1-\alpha$, $m_\alpha(X_1,\dotsc,X_\N)$ should be greater than or equal to the mean. 
Critically, in this statement the random quantity is the high-confidence upper bound, not the mean. 
Any definition of $m_\alpha$ that satisfies \eqref{eq:mainResult} for all bounded distributions,  samples sizes $n\geq 1$, and $1-\alpha \in [0,1]$, is said to have \emph{guaranteed coverage}.

In this paper we present a new method for constructing confidence intervals on the mean: a new $m_\alpha$. 
We conjecture that our function $m_\alpha$  satisfies \eqref{eq:mainResult}, i.e., that it has guaranteed coverage. 
If our conjecture holds, this is the first confidence interval with tightness comparable to the Student-$t$ interval, but with guaranteed coverage in this setting. 
We prove in Section~\ref{sec:related_work} that it {\em dominates} several other known confidence intervals with guaranteed coverage. 
That is, for every possible sample $(x_1,x_2,...,x_n)$, it produces a confidence interval with width less than or equal to these previous methods, often with a much smaller width. This makes our confidence interval suitable for small sample sizes where other methods are not practical. 

After defining $m_\alpha$ in the next section, we present the following results:
\begin{itemize}
    \item a proof of 
    \eqref{eq:mainResult} for a class of distributions that includes Bernoulli distributions, for all samples sizes $\N$, and for all confidence levels $1-\alpha$;
    \item the sketch of a proof that our intervals are always at least as tight as those provided by \citet{Anderson1969}, which, in turn, are strictly tighter than those of \citet{Hoeffding1963};
    \item results of extensive simulations on a wide variety of distributions that are consistent with
    \eqref{eq:mainResult} for many sample sizes and confidence intervals;
    \item empirical comparisons (through Monte Carlo simulations) with previous methods, demonstrating that the confidence intervals produced by $m_\alpha$ are consistently tighter than or as tight as the intervals produced by existing methods.
\end{itemize}

\section{A New Confidence Interval for the Mean}

In this section we present our new confidence interval. We also present our conjecture that it holds for all distributions bounded on $[0,1]$, for all sample sizes, and for all confidence levels. 

Let $\X\definedas (X_1,X_2,\dotsc,X_\N)$. 
Let $\Z\definedas (Z_1,Z_2,\dotsc,Z_\N)$ be the \emph{order statistics} of $\X$, i.e., $\Z$ is a vector containing the sorted values of $\X$ such that $Z_1\leq Z_2\leq\cdots\leq Z_\N$. 
Let $\z\definedas (z_1,z_2,...,z_\N)$ denote a particular sample of $\Z$ and $\x\definedas (x_1,x_2,\dotsc,x_\N)$ a sample of $\X$. 
For notational convenience, we alternate between viewing $m_\alpha$ as a function of $\z$ or $\mathbf x$. 
So, when we write $m_\alpha(\z)$ subsequently, this corresponds to a definition of $m_\alpha(\mathbf x)$ where $\z$ are the order statistics of $\mathbf x$.

Let $\U$ 
be the order statistics of a sample of size $\N$ from the continuous uniform distribution on $[0,1]$,
with $\uu\definedas (u_1,u_2,\dotsc,u_\N)$ being a particular sample of $\U$. 
Since $\uu$ are order statistics, 
$0\leq u_1\leq u_2 \leq ...\leq u_\N \leq 1$. 
We define a function of two ordered vectors: 
\begin{equation}
m(\z,\uu)\definedas 1-\sum_{i=1}^\N u_i(z_{i+1}-z_i),
\end{equation}
where $z_{n+1}\definedas 1$. 
Let $Q(1-\alpha,Y)$ be the {\em quantile function} of the scalar random variable $Y$, i.e.,
\begin{equation}
    \label{eq:quantileDefn}
    Q(1-\alpha,Y)\definedas \inf \{y\in \mathbb{R}: F_Y(y)\geq 1-\alpha\},
\end{equation}
where $F_Y(y)$ is the cumulative distribution function (CDF) of $Y$.

Consider the random quantity $m(\z,\U)$, which depends upon a fixed sample $\z$ (non-random) and
also on the random variable $\U$. 
We define $m_\alpha(\z)$ to be the $(1-\alpha)$-quantile of 
$m(\z,\U)$, i.e.,
\begin{equation}
    \label{eq:randomEq1}
m_\alpha(\z) \definedas Q(1-\alpha,m(\z,\U)).
\end{equation} 
We conjecture that this definition of $m_\alpha$ satisfies \eqref{eq:mainResult}: 

\begin{center}
\noindent\fbox{\begin{minipage}{\textwidth-1cm}
    \begin{conjecture}
    \label{thm:TLM_upper}
    Let $\X=(X_1,\dotsc,X_\N)$ be $\N$ independent and identically distributed random variables bounded in the interval $[0,1]$, each with mean $\mu$. Let $\Z=(Z_1,\dotsc,Z_\N)$ be the order statistics of $\X$. Then for all $\alpha \in [0,1]$\emph{:}
    \begin{equation}
    \Pr\left ( m_\alpha(\Z) \geq \mu \right ) \geq 1-\alpha,
\end{equation} 
where $m_\alpha$ is defined in \eqref{eq:randomEq1}.
\end{conjecture}
    \end{minipage}}
    \end{center}

Several extensions of this conjecture are apparent. 
First, since each $X_i$ is bounded above by $1$, this conjecture implies that $1-m_\alpha(1-\z)$ is a $(1-\alpha)$-confidence \emph{lower} bound on $\mu$. 
Second, if our main conjecture holds, we further conjecture that the assumption that the random variables are in $[0,1]$ can be extended to $(-\infty,1]$, or $[0,\infty)$ for the high-confidence lower bound. 
Furthermore, the deterministic upper bound of $1$ can be loosened to only require an almost-sure upper bound of $1$. 
Although these extensions may be important for some applications, hereafter we focus on the basic setting introduced previously.

\section{Understanding $m_\alpha(\z)$}
Our high-confidence bound (for brevity, hereafter we refer to it as simply our bound) is given by the function $m_\alpha$ defined above. In this section we introduce the following concepts, which provide intuition for $m_\alpha$: 
\begin{itemize}
    \item {\em ordered CDF  pairs},
    \item the {\em conservative completion} of a set of ordered CDF pairs,
    \item the {\em induced mean} of a set of ordered CDF pairs, via conservative completion. 
\end{itemize}
\subsection{Ordered CDF pairs}
For any order statistic vector $\z$, each element of $\z$ can be paired with an element from a non-decreasing sequence of numbers, $u_1, u_2, ..., u_\N$, to form $\N$ pairs:
\begin{equation}
(z_1,u_1),(z_2,u_2),...,(z_\N,u_\N).
\end{equation}
Assuming the $u$'s are all in the interval $[0,1]$ (as is the case if $\uu$ is a sample of $\U$), these pairs can be viewed as points on a CDF $F$, i.e., $u_i=F(z_i)$. 
For this reason, we refer to these $n$ pairs as \emph{ordered CDF pairs}, and write $(\z,\uu)$ to denote such a set of ordered CDF pairs.  
We say that a set of ordered CDF pairs is \emph{consistent} with a CDF $F$ if $u_i=F(z_i)$ for all $i \in \{1,2,\dotsc,\N\}$. 
Notice that a set of ordered CDF pairs is consistent with many (usually infinitely many) different CDFs---all non-decreasing functions on the interval $[0,1]$ that pass through these $n$ points (see Figure~\ref{fig:orderedPairs}). 

\begin{figure}[htbp]
 \begin{center}
   \includegraphics[trim=0 100 0 160,clip,width=5.in]{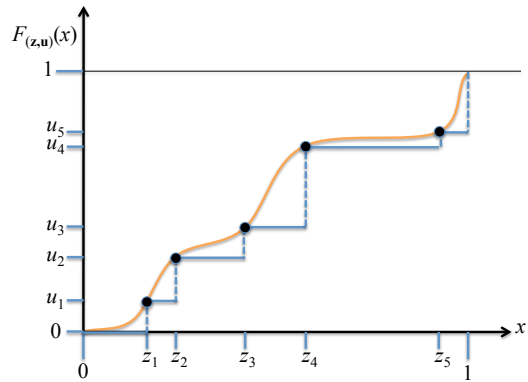}
    \caption{\label{fig:conservativeCompletion} Given a sample $\z$ and a vector $\mathbf{u}$ of sorted uniform samples, the ordered CDF pairs (black points) are compatible with a large family of CDFs. Two of the CDFs compatible with these points are shown, a smooth orange one, and a stairstep blue one. The blue one represents the CDF $F_{(\z,\uu)}(x)$, which has the greatest mean among all such CDFs, since it puts mass ``as far right'' as possible in a way that is still compatible with the ordered CDF pairs. We refer to this CDF as the {\em conservative completion} of the ordered CDF pairs $(\z,\mathbf{u})$. }
    \label{fig:orderedPairs}
\end{center}
\end{figure}
\subsection{Conservative Completion of Ordered CDF Pairs}

Given a set of ordered CDF pairs, one may ask which of the (usually infinitely many) CDFs that are consistent with the ordered CDF pairs represents the distribution with the greatest mean, and is this CDF unique?
This CDF \emph{is} unique, and we refer to it as the \emph{conservative completion} of the ordered CDF pairs. 
That is, the mean of the distribution characterized by the conservative completion represents an upper bound on the mean of any distribution consistent with the set of ordered CDF pairs.

The conservative completion for a set of ordered CDF pairs, $(\z,\mathbf{u})$, is illustrated in Figure~\ref{fig:conservativeCompletion}. 
It is given by the CDF:
\begin{equation}
F_{(\z,\uu)}(x)\definedas
\begin{dcases}
    0,& \text{for } x < z_1\\
    u_1,& \text{for } z_1 \leq x < z_2\\
    u_2,& \text{for } z_2 \leq x < z_3\\
    ...,& ...\\
    u_\N,& \text{for } z_\N \leq x < 1\\
    1,& \text{for } x\geq1.
\end{dcases}
\end{equation}

\begin{figure}[htbp]
 \begin{center}
   \includegraphics[trim=0 100 0 160,clip,width=5.in]{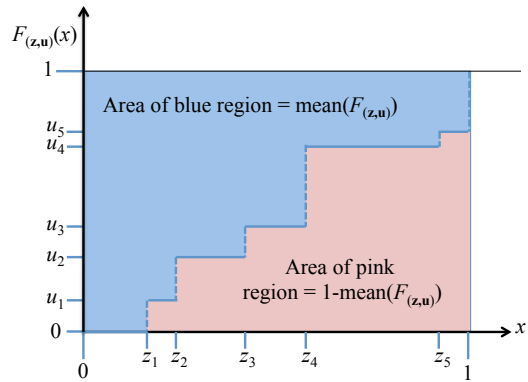}
    \caption{\label{fig:inducedMean} For CDFs defined on the interval $[0,1]$, the mean of the distribution characterized by $F_{(\z,\uu)}(x)$ is given by the area of the region above the CDF (blue), or one minus the area of the region below the CDF (pink).}
\end{center}
\end{figure}

\subsection{The Induced Mean, $m(\z,\uu)$}
We introduce $m(\z,\uu)$ to represent the mean of the distribution characterized by $F_{(\z,\uu)}(x)$. 
This quantity is, for distributions over $[0,1]$, equivalent to the area of the region above the CDF, as depicted in Figure~\ref{fig:inducedMean}.  The geometry of this figure suggests two methods for calculating this mean. The first, derived by decomposing the blue in Figure~\ref{fig:inducedMean} into a set of horizontal strips, is
\begin{equation}
m(\z,\uu)=\sum_{i=1}^{\N+1} z_i(u_i-u_{i-1}),
\end{equation}
where $u_0 \definedas 0$, $u_{\N+1}\definedas 1$, and $z_{\N+1}\definedas1$. 
Another formula is given by dividing the pink region into 
vertical strips, and is given by 
\begin{equation}
\label{eq:UdotS}
m(\z,\uu)=1-\sum_{i=1}^\N u_i(z_{i+1}-z_i).
\end{equation}
Next, we consider the distribution of such means obtained by allowing $\uu$ to vary in a particular fashion.
\subsection{A Distribution of Induced Means}

Recall that $\U$ is a random vector containing $\N$ samples from the continuous uniform distribution on $[0,1]$, sorted such that $0\leq U_1\leq U_2\leq ...\leq U_\N\leq 1$.
We now consider a distribution of induced means obtained by replacing the fixed $\uu$ in $m(\z,\uu)$ with a random vector $\U$ to form a new scalar random variable $m(\z,\U)$.

Recall the definition of the quantile function from \eqref{eq:quantileDefn}. 
We define $m_\alpha(\z)$ to be the $(1-\alpha)$-quantile of the random variable $m(\z,\U)$, i.e.,
\begin{equation}
m_\alpha(\z) \definedas Q(1-\alpha,m(\z,\U)).
\end{equation}
Thus, $m_\alpha(\z)$ considers the set of all $\uu$'s that can be used to form an ordered CDF pair with a particular sample $\z$ and chooses the $(1-\alpha)$-quantile of the resulting induced means.
As we shall see, this turns out to be just ``conservative enough'' to provide a valid high-confidence bound for Bernoulli-like distributions, and appears to be looser for distributions that are not Bernoulli-like.

\section{The Order Statistic Simplex and Feasible Set}
In this section, we define the {\em order statistic simplex} and the notion of {\em feasible} and {\em infeasible sets} of samples of order statistics. These definitions will be used in Section~\ref{sec:Bernoulli} to prove that our bound holds for all Bernoulli distributions and also for a more general set of Bernoulli-like distributions.

\subsection{A Conditional Analysis}
Our general method of proof (in the next section) will rely on a conditional analysis for a specific set of distributions. In particular, we will analyze our bound specifically for 
\begin{itemize}
    \item a fixed sample size $\N$,
    \item a subset of the distributions with a specific mean, $\mu$,
    \item a specific confidence level, $1-\alpha$ (or equivalently, failure rate $\alpha$).
\end{itemize}
If we can show that the bound,~\eqref{eq:mainResult}, holds for each tuple $(\N,\mu,1-\alpha)$, then we have a complete proof for the set of distributions under consideration.

Before proceeding with this method of proof, we define a few necessary terms.
\subsection{The Order Statistic Simplex}
\label{sec:orderStatSimplex}
Consider the {\em order statistic simplex} in $\N$ dimensions---the set of all possible order statistic vectors, $\z$, which forms a polytope of dimension $\N$ with $\N+1$ vertices, i.e., a simplex.
For distributions on $[0,1]$, we define the order statistic simplex as:
\begin{equation}
\mathcal{Z}\definedas \big \{\z=(z_1,z_2,...,z_\N): 0\leq z_1\leq z_2...\leq z_\N\leq 1 \big \}.
\end{equation}
The order statistic simplex for $\N=2$ is depicted by the blue region in Figure~\ref{fig:orderStatSimplex}.

\begin{figure}
 \begin{center}
    \includegraphics[trim=0 100 0 160,clip,width=6.in]{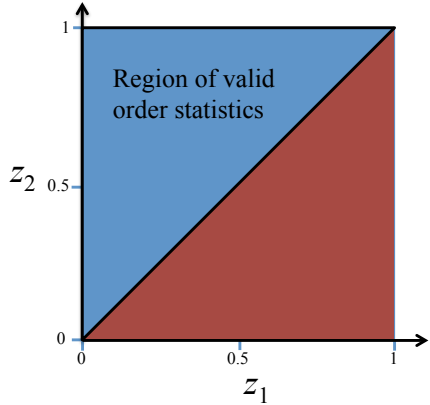}
    \caption{\label{fig:orderStatSimplex} The order statistic simplex in $\N=2$ dimensions. The upper left region (blue) shows the region of possible (valid) order statistics for a sample of size $\N=2$. For other $\N$, this region is defined by a polytope of $\N$ dimensions and $\N+1$ vertices, i.e., a simplex. We refer to this as the {\em order statistic simplex} in $\N$ dimensions. }
\end{center}
\end{figure}

\begin{figure}
 \begin{center}
 \vspace{-1in}
   \includegraphics[width=3.8in,height=3.8in]{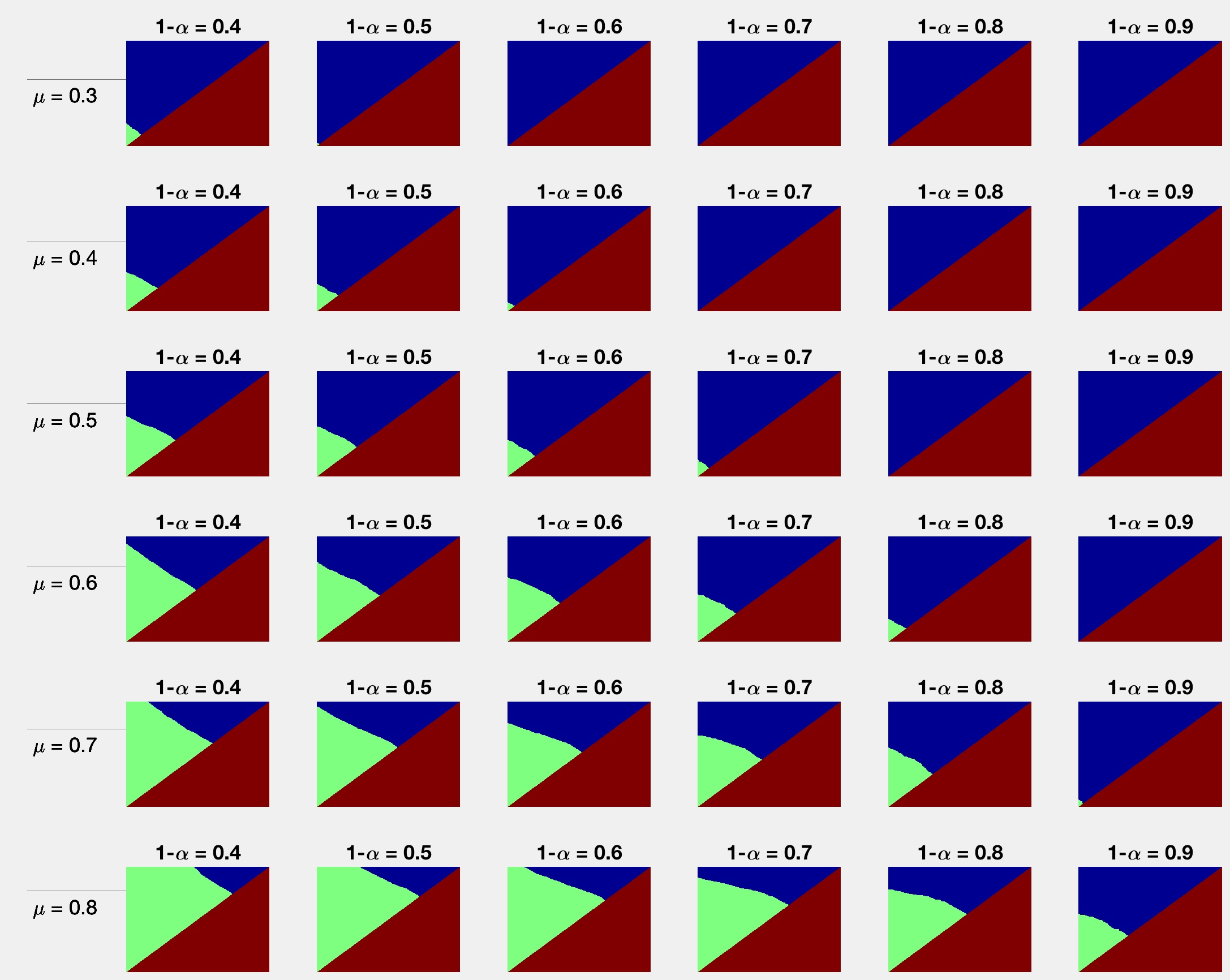}
    \caption{\label{fig:znot} Feasible and  
    infeasible regions for $\N=2$. For various values of the confidence level $1-\alpha$ and the mean $\mu$, we show the {\bf feasible regions} in blue (for which the bound is greater than or equal to the mean) and the {\bf infeasible regions} in green (for which the bound is less than the mean).}
\end{center}
\end{figure}
\subsection{The Infeasible Set of $\z$'s}
Let the sample size, $\N$, be fixed. 
For distributions with a specific mean, $\mu$, and for a specific confidence level, $1-\alpha$, let $\mz \subseteq \mathcal{Z}$ denote the set of $\z$'s for which $m_\alpha(\z)\geq\mu$. 
We refer to $\mz$ as the \emph{feasible set}, and say that $\z$ \emph{satisfies the bound} if $\z \in \mz$. 
Let $\nz$ denote the complement of this set, the set of $\z$'s for which $m_\alpha(\z)<\mu$.
We refer to $\nz$ as the {\em infeasible set} of $\z$'s, and say that $\z$ \emph{does not satisfy the bound} if it is in $\nz$. 
Note that, given the mean, $\mu$, of a distribution, the feasible and infeasible sets have no other dependency on the unknown distribution of $\X$. 

These ideas are illustrated in Figure~\ref{fig:znot}. For sample size $\N=2$, each plot shows the feasible (blue) and infeasible (green) regions for given confidence levels $1-\alpha$ and means $\mu$. Each row shows results for the same $\mu$ and different confidence levels. Note that in some cases, such as $1-\alpha= 0.6$ and $\mu=0.3$, the entire order statistic simplex is feasible (there are no green pixels).

\section{Bernoulli and Half-Bernoulli Distributions}
\label{sec:Bernoulli}
In this section, we present a proof of our conjecture for Bernoulli distributions and for a generalization of Bernoullis that we refer to as {\em half-Bernoulli} distributions. 
While Bernoulli distributions have point masses on both 0 and 1, half-Bernoullis can have point masses at two positions: $k$ and $1$, where $0\leq k <1$. 
Thus, they are a generalization of Bernoullis that allow the lower value to be any non-negative value less than 1. 
Let $H_{k,\mu}$ be the half-Bernoulli distribution where the point masses are at $k$ and $1$, and the mean is $\mu$. 
With $k$ and $\mu$ specified by $H_{k,\mu}$, the probability $p_k$ of sampling $k$ is given by 
\begin{equation}
    p_k=\frac{1-\mu}{1-k}.
\end{equation}

Bernoullis (and half-Bernoullis) are prime candidates for distributions for which the bound will fail (violate \eqref{eq:mainResult}), since it is not uncommon to have a sample of all 0's (or all $k$'s), despite having a relatively large mean. For example, the probability of obtaining a sample of $[0,0,0,0]$ from a Bernoulli distribution with parameter $p=0.5$ is $0.5^4=0.0625$. Since the probability of getting this sample is greater than $0.05$, the bound must produce a  result greater than $\mu=0.5$ for this sample at the confidence level $1-\alpha=0.95$.
As we demonstrate below, the bound holds for all Bernoulli and half-Bernoulli distributions. 

\subsection{Finding Worst Case Half-Bernoullis for $(\N,\mu,1-\alpha)$}
Our approach will be to find ``worst case'' distributions  among the set of half-Bernoulli distributions. 
By worst case, we mean that the probability of drawing a sample $\z$ for which the bound fails is as high as possible. 
We will show that for the worst case half-Bernoulli distributions (there can be more than one of these for each tuple $(\N,\mu,1-\alpha)$, the probability of drawing a sample $\z\in \nz$ is no more than $\alpha$. 
Since the bound holds for the worst case half-Bernoullis, it must hold for all half-Bernoullis.

\subsubsection{Enumerating possible $\z$'s}

Consider a half-Bernoulli distribution with probability masses at $k$ and $1$, and with mean $\mu$. 
For a given $\N$, there are $\N+1$ possible order statistics, $\z$:
\begin{align}
&[k, k, \dotsc, k, k]\\
&[k, k, \dotsc, k, 1]\\
&[k, k, \dotsc, 1, 1]\\
& \hspace{1cm}\vdots\\
&[k, 1, \dotsc, 1, 1]\\
&[1, 1, \dotsc, 1, 1].
\end{align}
Let $\z_{j,\N}$ be the sample $\z$ with $j$ out of $\N$ values of $k$, i.e., for $j \in \{0,1,\dotsc,\N\}$:
\begin{equation}
    \z_{j,\N}\definedas[\underbrace{k,\dotsc,k}_{j},\underbrace{1,\dotsc,1}_{n-j}].
\end{equation}
\subsubsection{Monotonicity of $m_\alpha(\z)$}
\label{sec:mono}
Notice that $m_\alpha$ is monotonic in the following sense. 
For two samples $\mathbf{y}$ and $\z$:
\begin{equation}
\label{eq:monotonic_bound}
{\tt If}\; \forall i, \, y_i\leq z_i \;{\tt then }\; m_\alpha(\mathbf{y})\leq m_\alpha(\z).
\end{equation}
It follows from \eqref{eq:monotonic_bound} that $m_\alpha(\z_{i,\N})\leq m_\alpha(\z_{j,\N})$ whenever $i<j$.

\subsubsection{An expression for $\Pr(\Z\in \nz)$}
For any half-Bernoulli distribution $H_{k,\mu}$, sample size $\N$, and confidence level $1-\alpha$, let
\begin{equation}
j_\text{min}(H_{k,\mu},1-\alpha,\N)=\min\big\{j \in \{0,1,\dotsc,\N\}: \z_{j,\N}\in \nz\big\},
\end{equation}
where $j_\text{min}(H_{k,\mu},1-\alpha,\N)=\N+1$ if $\z_{\N,\N} \in \mz$. 

For example, suppose $\N=5$ and $[k,k,k,1,1]\in \nz$ but $[k,k,1,1,1]\in \mz$. Then $j_\text{min}=3$, where here and in the following the arguments of $j_\text{min}$ are implicit. By the monotonicity of the bound 
(see \eqref{eq:monotonic_bound}),
all of the samples with ${\tt count}(k)\geq j_\text{min}$ will be in $\nz$. 

For a given half-Bernoulli distribution we can now write an expression for the probability that $\Z$ will not satisfy the bound:
\begin{align}
\Pr_{H_{k,\mu}}\left (\Z \in \nz\right ) =& 
\sum_{i=j_\text{min}}^\N \Pr_{H_{k,\mu}} \left (\Z=\z_{i,n}\right )\\ 
=& \sum_{i=j_\text{min}}^\N  \operatorname{Binomial}(i; \N, p_k)\\
\label{eq:ProbAsBeta}
=& \beta_\text{cdf}(p_k; j_\text{min},\N-j_\text{min}+1),
\end{align}
where $\beta_\text{cdf}(x;a,b)$ is the CDF of a beta distribution with parameters $a$ and $b$. 
The above derivation uses the property that each $\z_{i,\N}$ can be viewed as a sample from a binomial distribution, 
and in the last step we use a well-known identity that relates the sum of binomials to the CDF of a beta distribution.

\subsubsection{Simplification  of $m(\z_{j,\N})$ due to the simple structure of $\z_{j,\N}$}
Before continuing with deriving the $p_k$ that maximizes the failure rate of the bound, we show how $m(\z,\U)$ simplifies for samples from half-Bernoulli distributions. 

Recall that our bound is a quantile of the function $m(\z,\U)$ with respect to the uniform random variable $\U$.  For samples of the form $\z_{j,\N}$, this function reduces to a particularly simple form:
\begin{eqnarray}
m(\z_{j,\N},\U) &=& 1-\sum_{i=1}^\N U_i((\z_{j,\N})_{i+1}-(\z_{j,\N})_i)\\
&=&1-[0,...,0,1-k,0,...0]\cdot \U\\
\label{eq:simple}
&=& 1-(1-k) U_j.
\end{eqnarray}
That is, with the exception of the $j^\text{th}$ term, all of the successive differences of $\z_{j,\N}$ are 0,\footnote{We can ignore the case where $j=\N$, since the bound is trivial for $\z_{\N,\N}=(1,1,...,1)$.} leaving us with a simple function of the $j^\text{th}$ order statistic, $U_j$. Later it will be useful to note that the $j^\text{th}$ order statistic when taking $\N$ samples from a uniform distribution is beta distributed with parameters $j$ and $\N-j+1$ \citep[Example 5.4.5]{casella2002statistical}.

\subsubsection{Choosing $p_k$ to maximize the failure rate}
\label{subsubsec:choosingpk}
For a fixed $(\N,\mu,1-\alpha)$, consider the set of distributions, $H_{k,\mu}$, with $j_\text{min}=j$, for some value $j$. 
We are interested in the $H_{k,\mu}$ that maximizes $\Pr(\Z \in\nz)$. 
Since we are only considering half-Bernoulli distributions with a particular mean, $\mu$, the entire distribution is specified if $p_k$ is specified, and so we solve for the $p_k$ that maximizes $\Pr(\Z \in\nz)$:
\begin{align}
\underset{p_k: j_\text{min} = j}{\arg \max} \;\Pr_{H_{k,\mu}}(\Z \in \nz)
\overset{\text{(a)}}{=}& \underset{p_k: j_\text{min}=j}{\arg \max}\; \beta_\text{cdf}(p_k; j,\N-j+1)\\
\overset{\text{(b)}}{=}& \underset{p_k: j_\text{min}=j}{\arg \max} \; p_k.
\end{align}
Step (a) follows from \eqref{eq:ProbAsBeta}. Step (b) follows since all beta CDFs are monotonic in their first argument. 
In other words, within the set of $H_{k,\mu}$ that have the same $\mu$ and  $j_\text{min}$, the failure rate of the bound (the probability that $\Z\in\nz$) is monotonic in $p_k$.

Although the failure rate is monotonic in $p_k$, this does not mean that the worst-case is when $p_k=1$, since this monotonicity result is restricted to the set of half-Bernoulli distributions with $j_\text{min}=j$. 
We therefore now solve for the maximum $p_k$ such that $j_\text{min}=j$ in order to obtain the half-Bernoulli distribution with mean $\mu$ that maximizes the failure rate:
\begin{eqnarray}
&& \max \Big \{p_k \in [0,1] :  j_\text{min}=j\Big \}\\
&=& \max \Big \{p_k \in [0,1] :  \left \{\z_{1,\N},\z_{2,\N},...,\z_{j-1,\N}\right \} \subseteq \mz,\\
&&\hspace{1.05in} \left \{\z_{j,\N},\z_{j+1,\N},...,\z_{\N,\N}\right \} \subseteq \nz \Big\} \\
&\overset{(a)}{=}&\max \left \{p_k\in [0,1]: \z_{j,\N}\in \nz \right \}\\
&=& \max \left \{p_k\in [0,1]: \Pr_U\big(m(\z_{j,\N},U)< \mu\big)\geq 1-\alpha\right \}\\
&\overset{(b)}{=}& \max \left \{p_k\in [0,1]: \Pr_U\big(1-(1-k)U_j< \mu\big)\geq 1-\alpha \right\}\\
&=& \max \left \{p_k\in [0,1]: \Pr_U\left(U_j > \frac{1-\mu}{1-k}\right)\geq 1-\alpha \right \}\\
&=& \max \left \{p_k\in [0,1]: \Pr_U\left(U_j > p_k\right)\geq 1-\alpha\right \}\\
&=& \max \left \{p_k\in [0,1]: 1-\Pr_U\left(U_j \leq p_k\right)\geq 1-\alpha \right \}\\
&\overset{(c)}{=}& \max \left \{p_k\in [0,1]: 1-\beta_\text{cdf}\left(p_k; j, \N-j+1\right)\geq 1-\alpha \right \}\\
&=& \max \left \{p_k\in [0,1]: \beta_\text{cdf}\left(p_k; j, \N-j+1\right)\leq \alpha \right\}\\
&=& \max \left \{p_k\in [0,1]: \beta_\text{cdf}^{-1}(\alpha; j,\N-j+1) \geq p_k \right \}\\
&=& \beta_\text{cdf}^{-1}\left (\alpha; j,\N-j+1\right ).
\end{eqnarray}
Step (a) follows from the monotonicity of the bound and Section~\ref{sec:mono}. Step (b) uses the result of \eqref{eq:simple}. Step (c) uses the the fact that the $j^\text{th}$ order statistic of a uniform sample of size $n$ is beta distributed
with parameters $j$ and $n-j+1$.

\subsubsection{Bringing the pieces together}

We have established that, for a given $\N$, of all half-Bernoulli distributions with mean $\mu$ and $j_\text{min}=j$, the one that maximizes the failure rate of the bound has $p_k=\beta_\text{cdf}^{-1}\left (\alpha; j,\N-j+1\right )$. 
Plugging this into \eqref{eq:ProbAsBeta}, we have that 
\begin{align}
\underset{H_{k,\mu}:j_\text{min}=j}{\max} \; \Pr_{H_{k,\mu}}(\Z \in \nz) 
=& \beta_\text{cdf}(\beta_\text{cdf}^{-1}(\alpha; j,\N-j+1); j,\N-j+1)\\
=& \alpha.
\end{align}

Thus, we have seen that by maximizing the probability that $\Z$ causes $m_\alpha(\Z) > \mu$, i.e., maximizing $\Pr(\Z \in \nz)$, we can produce a probability of violation of at most $\alpha$. 
Thus, the bound holds with probability at least $1-\alpha$. 
Since this is true for all half-Bernoulli distributions for all values of $j_\text{min}$ and arbitrary tuples $(\N, \mu, 1-\alpha)$, it is true for all half-Bernoullis, all sample sizes and all confidence levels.

\section{Computing $m_\alpha$}
In this section, we discuss two methods for computing our bound, $m_\alpha(\z)$, for a particular sample $\z$. The first is based upon a geometric analysis of the bound and the second uses a Monte Carlo sampling technique.

\subsection{Geometric Computation of the Bound}
Recall that the random variable $\U$ represents the order statistics of a uniform sample, and hence lies in the order statistic simplex defined in Section~\ref{sec:orderStatSimplex}. Figure~\ref{fig:simplexVol} shows the order statistic simplex for $n=2$. Note that the order statistic simplex of dimension $n$ has volume $\frac{1}{n!}$.

Let $\mathbf{s}\definedas [z_2-z_1,z_3-z_2,...,z_{n+1}-z_n]$ be the {\em spacings} of the sample. 
Consider an example in which $\z=[0.3, 0.8]$. Then $\mathbf{s}=[0.5, 0.2]$, as shown in the figure. 

\begin{figure}
 \begin{center}
   \includegraphics[trim=0 100 200 210,clip,width=4.in]{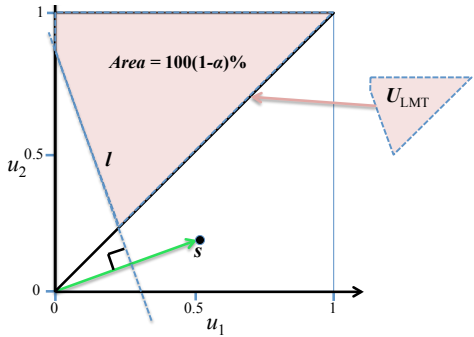}
    \caption{\label{fig:simplexVol}  The figure shows several quantities related to the geometric computation of our bound. The upper left triangle represents the order statistic simplex. The point $\mathbf{s}$ represents the spacings of the sample $\z$. The pink region, which we define later to be $\mathcal U_\text{LMT}$, is a section of the order statistic simplex cut by the hyperplane $l$, which represents the set of vectors $\uu$ for which $\uu \cdot \mathbf{s}$ is greater than or equal to some value. The goal is to find the maximum such value, and thereby the minimum $\hat{\mu}$ such that the volume of the pink region is $100(1-\alpha)\%$ of the volume of the order statistic simplex. 
    }
\end{center}
\end{figure}Starting from the definition of our bound $m_\alpha(\z)$ and expanding the definition of the quantile function, we have
\begin{align}
m_\alpha(\z)
=&\inf\left \{\hat{\mu}\in \mathbb R : \Pr(m(\z,\U) \leq \hat{\mu}) \geq 1-\alpha\right \}\\
=&\inf\left \{\hat{\mu} \in \mathbb R: n! \operatorname{Volume}(\{\uu: m(\z,\uu) \leq \hat{\mu}\}) \geq 1-\alpha\right \}\\
=&\inf\left \{\hat{\mu}\in \mathbb R : n! \operatorname{Volume}(\{\uu: 1-\uu\cdot \mathbf{s} \leq \hat{\mu}\}) \geq 1-\alpha\right \}\\
=&\inf\left \{\hat{\mu}\in \mathbb R : n! \operatorname{Volume}(\{\uu: \uu\cdot \mathbf{s} \geq 1-\hat{\mu}\}) \geq 1-\alpha\right \}.
\end{align}
This final expression has a clear geometric interpretation. The set of points $\uu$ such that $\uu\cdot \mathbf{s}$ is greater than some value is the upper right region of the order statistic simplex, depicted by the pink region in Figure~\ref{fig:simplexVol}. The bound is defined to be the least value of $\hat{\mu}$ such that the volume of the pink region is a fraction $1-\alpha$ of the simplex volume. 

This value of $\hat{\mu}$ can be found by evaluating the volume of the section of the order statistic simplex above 
the hyperplane $l$, a hyperplane orthogonal to the spacings vector $\mathbf{s}$. Thus, we seek the smallest value of $\hat{\mu}$ such that the volume of this section 
is $100(1-\alpha)\%$ of the volume of the order statistic simplex. This value of $\hat{\mu}$ is our bound.

Closed-form expressions for sections of simplexes cut by hyperplanes have been published by several authors, including \citet{lasserre2015volume}. 
These expressions lead to efficient calculations of the bound in most cases. However, these formulas have singularities that cause problems for certain samples $\z$, such as samples with repeated values. Thus, we explore a more reliable Monte Carlo approach for computing our bound below.

\subsection{Monte Carlo Estimate of the Bound}
Since the bound is defined in terms of a quantile of a function that depends upon a uniform random variable, it is  simple to develop a Monte Carlo estimate.  This is provided in Algorithm~\ref{alg:theAlg}.

\RestyleAlgo{boxruled}
\LinesNumbered
\begin{algorithm}[H]
 \SetKwInOut{Input}{Input}
 \SetKwInOut{Output}{Output}
 \Input{A sample $\x$, confidence parameter $\alpha<1$, and Monte Carlo sampling parameter $l$.}
 \Output{An estimate of $m_\alpha(\x)$.}
 $n \gets \operatorname{length}(\x)$\;
 $\z \gets \operatorname{sort}(\x,\text{ascending})$\;
 Create array $\mathbf{ms}$ to hold $l$ floating point numbers, and initialize it to zero\;
 Create arrays $\uu$ and $\mathbf s$, each to hold $n$ floating point numbers\;
 \For{$i \gets 1$ \KwTo $n-1$}{
    $\mathbf s[i] = z[i+1]-z[i]$\;
 }
 $\mathbf s[n] \gets 1-z[n]$\;
 \For{$i\gets1$ \KwTo $l$}{
    \For{$j\gets 1$ \KwTo $n$}{
        $\uu[j] \sim \operatorname{Uniform}(0,1)$\;
    }
    $\operatorname{sort}(\uu,\text{ascending})$\;
    $\mathbf{ms}[i] \gets 1-\mathbf s\cdot \uu$\;
 }
 $\operatorname{sort}(\mathbf{ms},\text{ascending})$\;
 \Return $\mathbf{ms}[\lceil (1-\alpha)l\rceil]$\;
 \caption{Monte Carlo Estimation of $m_\alpha$\newline This pseudocode uses one-based indexing of arrays.}
 \label{alg:theAlg}
\end{algorithm}

Algorithm \ref{alg:theAlg} can be implemented more efficiently if $m_\alpha$ will be estimated multiple times for the same $\N$, since samples of $\U$ can be computed and sorted a single time. 
Also, notice that the number of Monte Carlo samples, $l$, does not scale poorly for distributions with rare values, since we are estimating the $(1-\alpha)$-quantile of $m(\z,\U)$, which is robust to outliers (e.g., $\alpha=0.5$ makes this the median, which is well known to be robust to outliers). 

In practice, we find that $l=10,\!000$ tends to provide a reasonable approximation of $m_\alpha(\x)$ for $\alpha=0.05$. 
Note that \eqref{eq:mainResult} may not hold when using this Monte Carlo estimate of $m_\alpha$ due to error in the estimate. 
In practice this can be remedied by increasing $l$, or by incorporating high-probability bounds on the error in the Monte Carlo estimate into the bound. 
Also, note that as $\alpha$ decreases, $l$ should be increased. 

\section{Related Work}
\label{sec:related_work}
In this section, we review other methods for computing high-confidence upper bounds on the mean of a random variable from samples (several of which we compare to in the subsequent numerical analysis section). 
Although some of these methods extend to more general settings (e.g., Hoeffding's inequality does not require identically distributed samples, and Anderson's inequality does not require a lower-bound on the random variable), here we consider only the standard setting that we have discussed in this paper, wherein the samples are i.i.d.~and the random variable always takes values in the interval $[0,1]$. 
We divide this section into two parts: prior methods that provide guaranteed coverage, and prior methods that do not provide guaranteed coverage. 
We present these prior methods as functions, $m_\alpha^\text{Hoeffding}$, $m_\alpha^\text{Maurer\&Pontil}$, etc., each of which provides an alternative to $m_\alpha$.

\subsection{Prior Methods with Guaranteed Coverage}

The methods presented in this subsection have guaranteed coverage in the setting that we have described---they satisfy \eqref{eq:mainResult} if used in place of $m_\alpha$. 

Using Hoeffding's inequality  \citep{Hoeffding1963} to construct a high-confidence upper-bound on $\mu$ is perhaps the best known, and simplest, prior method with guaranteed coverage:
\begin{equation}
    m_\alpha^\text{Hoeffding}(\x) \definedas \bar \x + \sqrt{\frac{\ln(1/\alpha)}{2\N}},
\end{equation}
where $\bar \x$ is the sample mean, i.e., $\bar \x \definedas \frac{1}{\N}\sum_{i=1}^\N x_i$. 
In cases where the variance of the random variable is significantly less than one, the upper bounds provided by Maurer and Pontil's empirical Bernstein bound \citep{Maurer2009} can be tighter than those produced by Hoeffding's inequality. 
This is achieved by leveraging not just the sample mean, $\bar \x$, but also the sample variance, $\widehat{\operatorname{Var}}(\x)\definedas \frac{1}{n-1}\sum_{i=1}^n (x_i-\bar \x)^2$: 
\begin{equation}
    m_\alpha^\text{Maurer\&Pontil}(\x)\definedas \bar \x +\sqrt{\frac{2 \widehat{\operatorname{Var}}(\x) \ln(2/\alpha)}{\N}} + \frac{7 \ln(2/\alpha)}{3(\N-1)}.
\end{equation}
Going one step further, Anderson's inequality provides high-confidence upper bounds on the mean by using the entire sample CDF (rather than only the sample mean and variance): 
\begin{equation}
    m_\alpha^\text{Anderson}(\z)\definedas
    m(\z, \uu^\text{DKW}),
\end{equation}
where for $i \in \{1,2,\dotsc,\N\}$, 
\begin{equation}
    \label{eq:udkw}
    \mathbf u^\text{DKW}_i\definedas \max \left \{0,i/n-\sqrt{\ln(1/\alpha)/2\N} \right \}
\end{equation} 
is a vector that Anderson derived from the Dvoretsky-Kiefer-Wolfowitz (DKW) inequality \citep{Dvoretzky1956b}. 
Note that the form we present for Anderson's inequality uses the tight constant for the DKW inequality found by \citet{Massart90}, which relies on the assumption that $\alpha \leq 0.5$. 
This restriction is not restrictive because high-confidence bounds are typically applied with small values of $\alpha$, e.g., $\alpha=0.05$.

Following these three methods, several alternatives have been proposed, including other approaches that rely only on the sample mean \citep{chen2008confidence}, methods that extend Maurer and Pontil's empirical Bernstein bound to provide tighter bounds for random variables with long tails \citep{Bubeck2012,Thomas2015}, and methods that provide alternatives to Anderson's inequality that use alternative methods of defining inclusion envelopes for a distribution's CDF~\citep{Learned-Miller2008,Diouf2005}.

\subsection{Prior Methods without Guaranteed Coverage}

All of the methods presented in this subsection do \emph{not} have guaranteed coverage in the setting that we have described, but are often used to compute high-confidence upper bounds on the mean.

Perhaps the most common method for constructing high-confidence upper bounds is based on Student's $t$-statistic \citep{student1908probable}:
\begin{equation}
    m_\alpha^\text{Student}(\x) \definedas \bar \x + \sqrt{\frac{\widehat{\operatorname{Var}}(\x)}{\N}}t_{1-\alpha,\N-1},
\end{equation}
where $t_{1-\alpha,\nu}$ denotes the $100(1-\alpha)$ percentile of the Student's $t$ distribution with $\nu$ degrees of freedom. 
We refer to this confidence interval as the Student-$t$ interval. 
If $\bar \X$ is normally distributed, then $m_\alpha^\text{Student}$ \emph{does} provide guaranteed coverage. 
The central limit theorem implies that $\bar \X$ tends towards a normal distribution as $\N$ increases, and so this method is often applied in scientific research if $\N \geq 30$, even though this does not provide  coverage guarantees.

Bootstrap methods tend to provide the tightest confidence intervals for the mean. 
However, this comes at a large cost: they do not have guaranteed coverage, even with normality assumptions. 
Despite concerns about their reliability, bootstrap methods remain in common use due to their tight confidence intervals and tendency to produce error rates roughly around $\alpha$ for many common distributions~\citep{hanna2017bootstrapping,Thomas2015c}. 
Since bootstrap methods are not easily expressed as closed-form alternatives to $m_\alpha$, we refer the reader to the work of \citet{Efron1993} for details on these approaches. 
The two that we focus on in our subsequent experiments are the most common, the percentile bootstrap, and one of the most sophisticated, the \emph{bias corrected and accelerated} (BCa) bootstrap.

One limitation of BCa, the more sophisticated bootstrap method, is that it is not defined in some cases (e.g., if all of the samples take the same value) and can encounter numerical issues in other cases. 
In our implementations, whenever numerical issues are detected, the method automatically reverts to the percentile bootstrap.

\section{Theoretical Analysis}

In this section, we provide an analytic comparison of our bound to two prior methods that provide guaranteed coverage: Hoeffding's inequality and Anderson's inequality. 
We will show that, for any sample $\z$, our high-confidence upper bound is less than that resulting from Hoeffding's inequality, and never greater than that of Anderson's inequality. 
We break this result into two components. 
First we show that for all $\z$, $m_\alpha(\z) \leq m_\alpha^\text{Anderson}(\z)$. 
Second, we show that for all $\z$, $m_\alpha^\text{Anderson}(\z) \leq m_\alpha^\text{Hoeffding}(\z)$, which implies that $m_\alpha(\z) \leq m_\alpha^\text{Hoeffding}(\z)$, where these inequalities are strict if $\alpha \leq 0.5$. 

\subsection{Theoretical Comparison to Anderson's Inequality}

In this section we compare $m_\alpha$ to $m_\alpha^\text{Anderson}$. 
\begin{theorem}
    \label{thm:LMTVsAnderson}
    For all possible values $\z$ of $\Z$ and all $\alpha \in [0,0.5]$, 
    \begin{equation}
        m_\alpha(\z) \leq m_\alpha^\text{Anderson}(\z).
    \end{equation}
\end{theorem}
\begin{proof}
We present a sketch of the proof. 
Consider the diagram in  Figure \ref{fig:dominance}. 
This figure depicts, for $\N=2$, the space of possible vectors $\uu$, which are sorted uniform samples. 
    The point $(u_1,u_2)$ represents $\uu^\text{DKW}$, defined in \eqref{eq:udkw}. 
    $\mathcal U_\text{DKW}$ denotes the set of vectors that are element-wise greater than $\uu^\text{DKW}$. 
    It follows from the DKW inequality, with the tight constants found by \citet{Massart90}, that the probability $\U$ is in $\mathcal U_\text{DKW}$ is at least $1-\alpha$. 
    The region $\mathcal U_\text{LMT}$ is any set of $\uu$'s that result in the lowest induced means, $m(\z,\uu)$, while ensuring that the probability that $\U$ is in $\mathcal U_\text{LMT}$ is precisely $1-\alpha$. 
    Note that any point that is not contained within the pink region must represent a vector $\uu$ that results in an induced mean, $m(\z,\uu)$, which is greater than the induced mean of any point in the pink region.  
    Our bound is effectively the maximum over induced means of points in the pink region and Anderson's bound is the maximum over points in the blue region. 
    Since the probability that $\U$ is in $\mathcal U_\text{DKW}$ cannot be less than the probability that $\U$ is in $\mathcal U_\text{LMT}$, and $\mathcal U_\text{LMT}$ contains the $\uu$ vectors that minimize the induced mean, our bound cannot be larger than Anderson's. 
\end{proof}

\begin{figure}[htbp]
 \begin{center}
   \includegraphics[trim=0 100 200 200,clip,width=4.in]{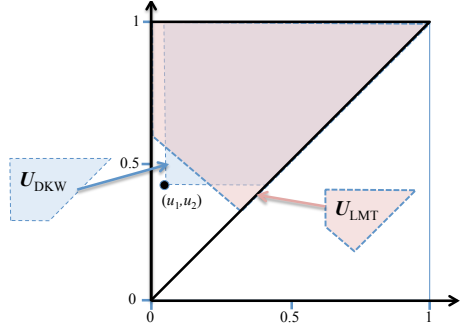}
    \caption{\label{fig:dominance}  Diagram comparing our bound and Anderson's.
    }
\end{center}
\end{figure}

\subsection{Analytic Comparison to Hoeffding's Inequality}

In this section we prove the following theorem:
\begin{theorem}
    \label{thm:AndersonVsHoeffding}
    For all possible values $\z$ of $\Z$ and all $\alpha \in [0,1]$, 
    \begin{equation}
        m_\alpha^\text{Anderson}(\z) \leq m_\alpha^\text{Hoeffding}(\z),
    \end{equation}
    where the inequality is strict if $\alpha \leq 0.5$.
\end{theorem}
\begin{proof}
We begin with $m_\alpha^\text{Anderson}(\z)$ and present a sequence of inequalities that conclude with $m_\alpha^\text{Hoeffding}(\z)$, where one inequality is strict if $\alpha \leq 0.5$: 
\begin{align}
    m_\alpha^\text{Anderson}(\z) =& m\left (\z,\uu^\text{DKW}\right )\\
    =& 1-\sum_{i=1}^\N (z_{i+1}-z_i) \uu_i^\text{DKW}\\
    =& 1-\sum_{i=1}^\N (z_{i+1}-z_i) \max\left \{ 0, \frac{i}{\N} - \sqrt{\frac{\ln(1/\alpha)}{2\N}} \right \}\\
    \leq &1-\sum_{i=1}^\N (z_{i+1}-z_i) \left (\frac{i}{\N} - \sqrt{\frac{\ln(1/\alpha)}{2n}}\right ). 
\end{align}
If $\alpha \leq 0.5$, then this final inequality is strict because, when $i=1$, we have that for any $n$,  $0 > i/n-\sqrt{\ln(1/\alpha)/2n}$, and so \begin{equation}
    \max\left \{0,\frac{i}{n} - \sqrt{\frac{\ln(1/\alpha)}{2n}}\right \} > \frac{i}{n}-\sqrt{\frac{\ln(1/\alpha)}{2n}}.
\end{equation}
Continuing, we have:
\begin{align}
    m_\alpha^\text{Anderson}(\z)
    \leq&1-\sum_{i=1}^\N (z_{i+1}-z_i)\frac{i}{n} + \sum_{i=1}^n (z_{i+1}-z_i) \sqrt{\frac{\ln(1/\alpha)}{2n}}\\
    =& 1+\frac{1}{n}\left (\sum_{i=1}^n i z_i - \sum_{i=1}^n i z_{i+1}\right ) + (z_{n+1}-z_1)\sqrt{\frac{\ln(1/\alpha)}{2n}}\\
    =& 1+\frac{1}{n}\left (\sum_{i=1}^n i z_i - \left (\sum_{i=2}^{n} (i-1) z_i \right )- n\right ) + (1-z_1)\sqrt{\frac{\ln(1/\alpha)}{2n}}\\
    \label{eq:lkjasdfasdf}=& \frac{1}{n} \sum_{i=1}^n z_i + (1-z_1)\sqrt{\frac{\ln(1/\alpha)}{2n}}\\
    \leq& \frac{1}{n} \sum_{i=1}^n z_i + \sqrt{\frac{\ln(1/\alpha)}{2n}}\\
    =& m_\alpha^\text{Hoeffding}(\z).
\end{align}
Notice that \eqref{eq:lkjasdfasdf} provides an expression similar to Hoeffding's inequality, but where the lower bound on the random variable (in our case, zero) is replaced by the smallest observed sample, $z_1$. 
This presents a tighter variant of Hoeffding's inequality that holds when $\alpha \leq 0.5$ and the random variables are i.i.d.~(the general form of Hoeffding's inequality holds for random variables that are not necessarily identically distributed). 
\end{proof}

It then follows from Theorem \ref{thm:AndersonVsHoeffding} that our bound is always at least as tight as Hoeffding's inequality, and is strictly tighter if $\alpha \leq 0.5$:
\begin{corollary}
    For all possible values $\z$ of $\Z$ and all $\alpha \in [0,1]$, 
    \begin{equation}
        m_\alpha(\z) \leq m_\alpha^\text{Hoeffding}(\z),
    \end{equation}
    where the inequality is strict if $\alpha \leq 0.5$.
\end{corollary}
\begin{proof}
This follows immediately from Theorems \ref{thm:LMTVsAnderson} and \ref{thm:AndersonVsHoeffding}.
\end{proof}

\section{Numerical Analysis}

In this section we present results from a numerical analysis of our bound. 
These empirical results aim to answer the following research questions:
\begin{enumerate}
    \item[RQ1] For a variety of distributions, confidence levels, and number of samples, are results consistent with~\eqref{eq:mainResult}?
    \item[RQ2] For a variety of distributions that resemble common use-cases, how do the confidence intervals produced by our bound compare to those of previous methods that have guaranteed coverage (i.e., those that satisfy \eqref{eq:mainResult})?
    \item[RQ3] This question is the same as RQ2, but for methods that do \textbf{not} have guaranteed coverage. 
    \item[RQ4] Can our bound provide confidence intervals that are practical for scientific experiments with fewer than $30$ samples?
\end{enumerate}

\subsection{Numerical Studies on Guaranteed Coverage}

In this subsection we study RQ1 with experiments that are consistent with the conjecture that $m_\alpha$, as we have defined it, has guaranteed coverage (satisfies \eqref{eq:mainResult}). 
Although they show that \eqref{eq:mainResult} appears to hold for a variety of settings, this does not imply that settings do not exist under which \eqref{eq:mainResult} does not hold.

To study RQ1, we selected a variety of different distributions (uniform, beta, and Bernoulli, each with various parameters), confidence levels $1-\alpha$, and number of samples, $\N$. 
For each such tuple, $(\text{distribution},1-\alpha,\N)$, we collected $10,\!000$ samples of $\Z$, computed $m_\alpha(\Z)$, and checked whether $m_\alpha(\Z) \geq \mu$. 
From these $10,\!000$ tests, we estimated the coverage---the probability that the bound holds. 
That is, we estimated $\Pr(m_\alpha(\Z) \geq \mu)$ by dividing the number of samples of $\Z$ such that $m_\alpha(\Z) \geq \mu$ by $10,\!000$.

Although or goal here is to study RQ1, to facilitate the interpretation of the presented results, we provide results using two prior methods that exhibit the different types of behavior that our bound might produce. 
This comparison also provides some insight into RQ2 (via the comparison to Heoffding's inequality) and RQ3 (via comparison to the Student-$t$ interval). 
However, note that subsequent experiments further study these two research questions.

First consider Figure \ref{fig:alphaPlots_Hoeffding}, which presents results using Hoeffding's inequality. 
The top left plot shows the coverage for a variety of beta distributions, but with $\N$ fixed. 
To interpret this plot, consider the curve $\beta(1,5)\,\,n=10$ at the $0.7$ position on the horizontal axis. 
The position $0.7$ on the horizontal axis indicates that we requested an upper bound that holds with probability at least $0.7$. 
Since, at horizontal position $0.7$, the curve $\beta(1,5)\,\,n=10$ (red curve) lies above $0.7$ (blue curve), the upper bound produced by Hoeffding's inequality held with probability greater than $0.7$. 
Hence, a curve remaining above the blue line indicates that the desired confidence level was achieved. 
However, notice that the curve is far above the blue line---this indicates that Hoeffding's inequality was overly conservative. 
It provided a high-confidence upper bound that was greater than or equal to the mean far more often than requested. 
This means that for this distribution the confidence interval provided by Hoeffding's inequality is not tight, and could be improved.

The three other plots in Figure \ref{fig:alphaPlots_Hoeffding} are similar, but use different parameters. 
The top row presents results for beta distributions, while the bottom row presents results for Bernoulli distributions. 
The left column presents results as the parameters of the distributions are varied, while the right column presents results as the number of samples is varied. 
Overall Figure \ref{fig:alphaPlots_Hoeffding} shows the behavior that we would expect of a bound that has guaranteed coverage, but which provides loose high-confidence upper bounds.

\begin{figure}[htbp]
    \centering
    \includegraphics[width=0.45\textwidth]{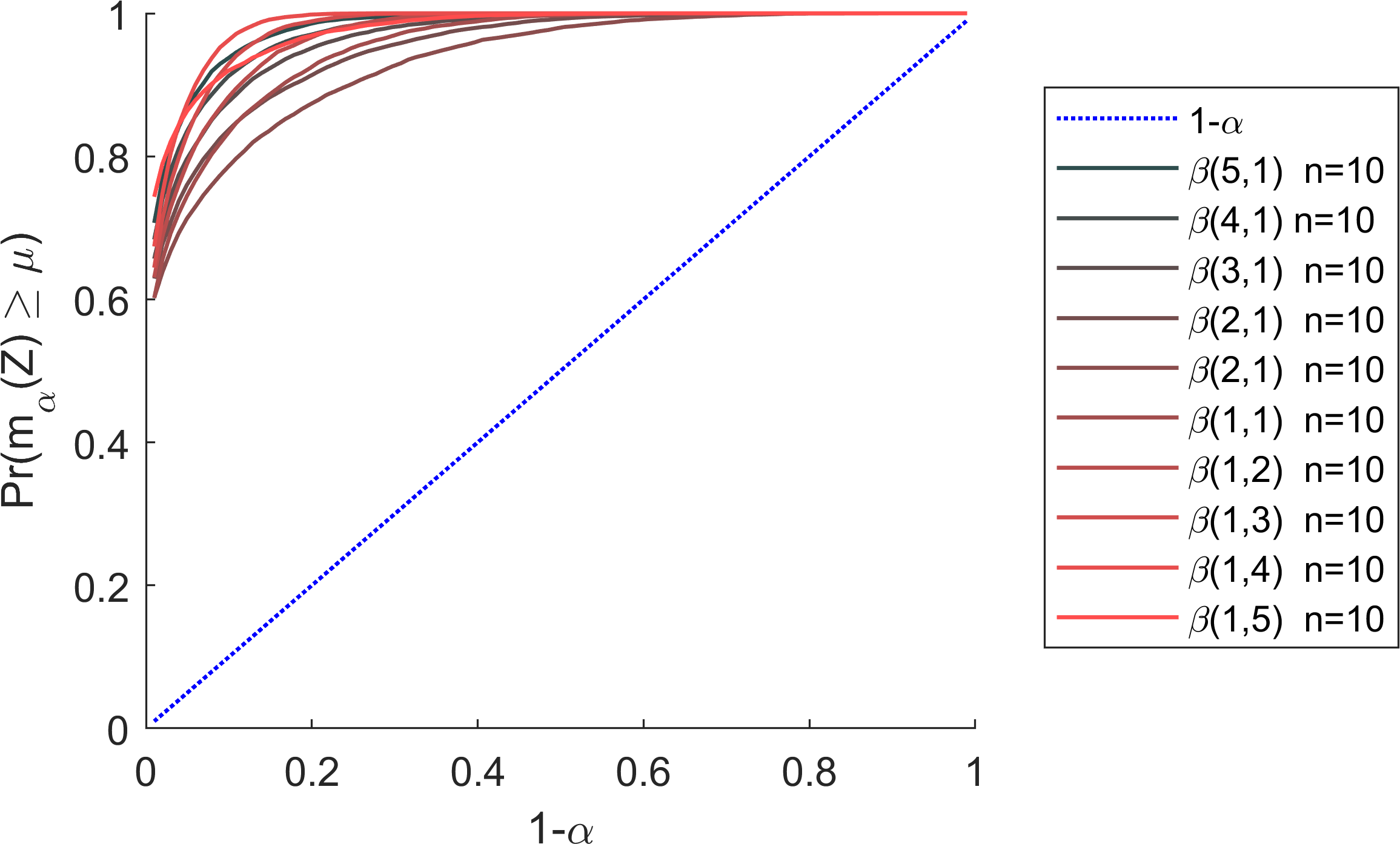}~
    \includegraphics[width=0.45\textwidth]{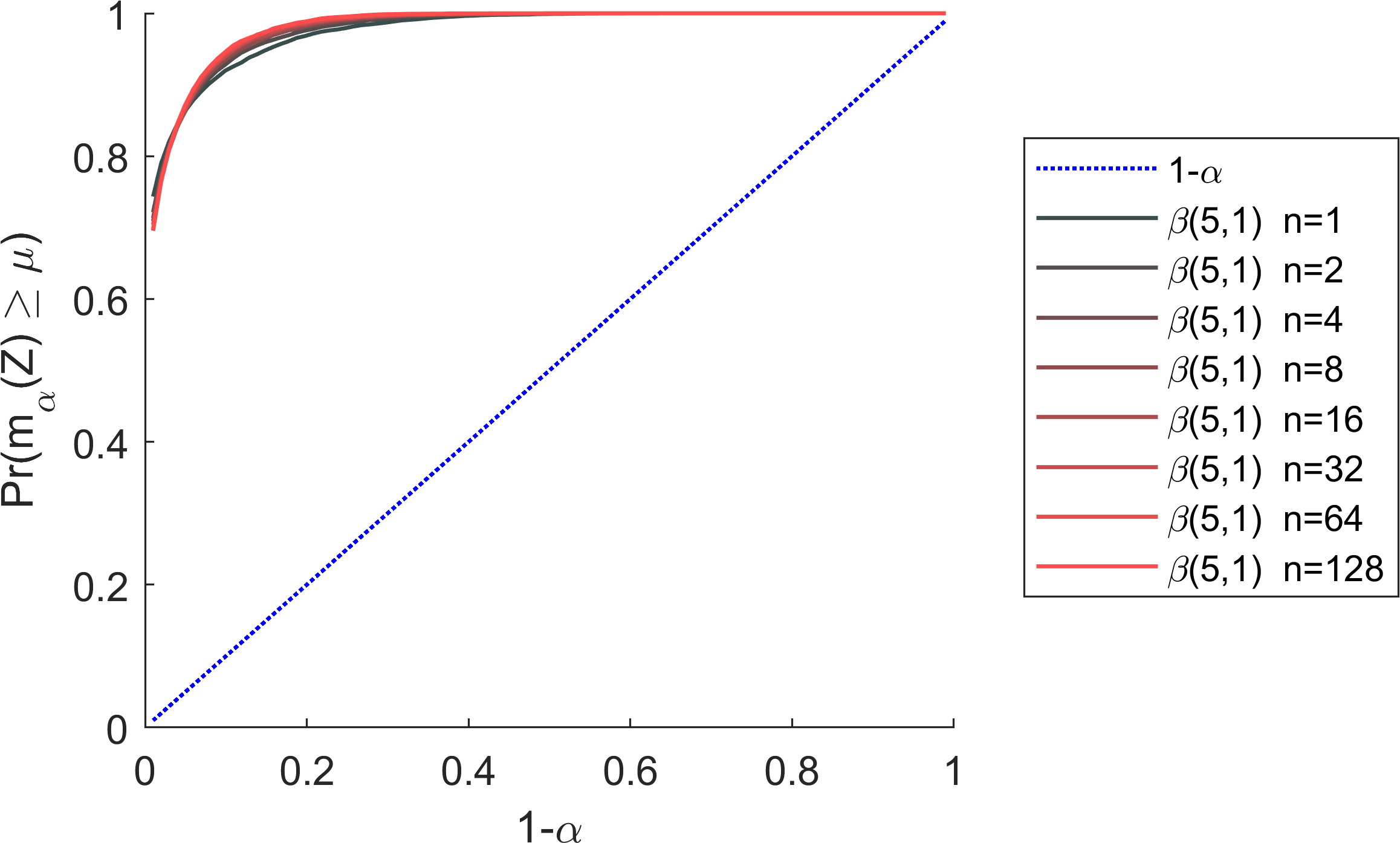}\\
    \includegraphics[width=0.45\textwidth]{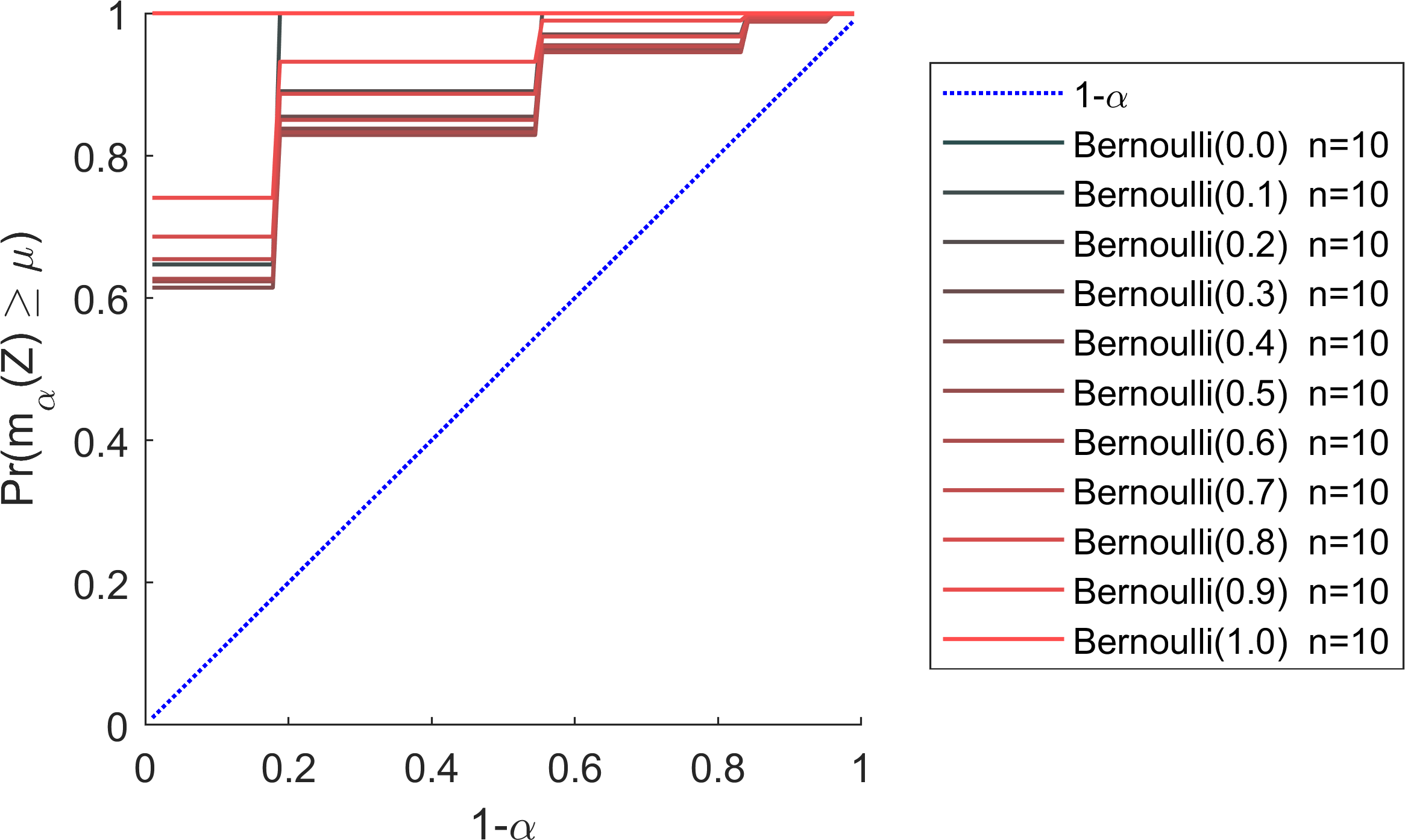}~
    \includegraphics[width=0.45\textwidth]{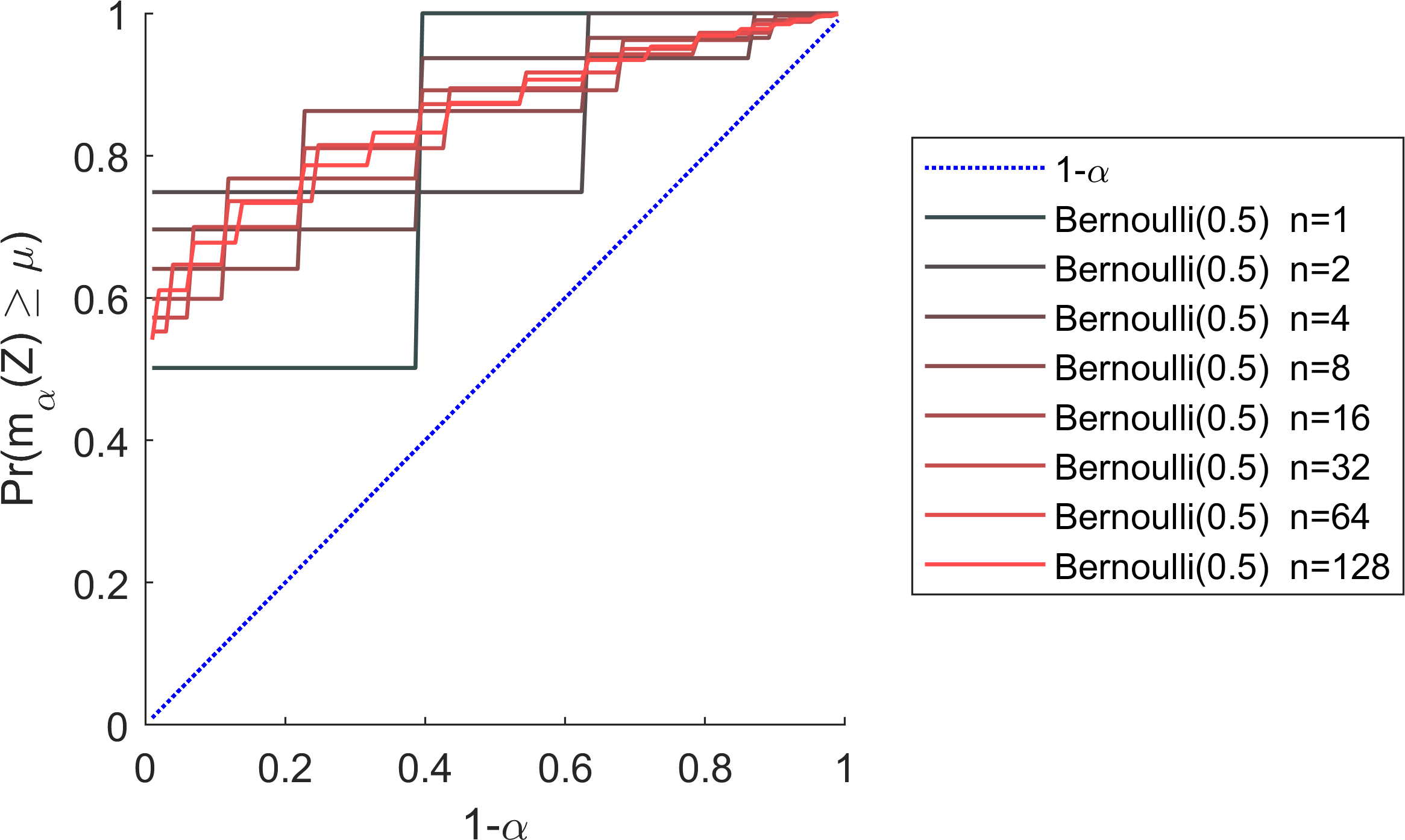}
    \caption{Estimated probability that the high-confidence upper bound produced using Hoeffding's inequality is greater than or equal to the true mean.}
    \label{fig:alphaPlots_Hoeffding}
\end{figure}

Now consider Figure \ref{fig:alphaPlots_ttest}, which is identical to Figure \ref{fig:alphaPlots_Hoeffding}, except that it uses the Student-$t$ interval instead of Hoeffding's inequality. 
This plots shows very different behavior: the curves tend to be much lower, indicating tighter confidence intervals around the sample mean. 
However, the curves often cross the blue line, indicating that in these settings the Student-$t$ interval does \emph{not} provide guaranteed coverage---if you ask for an $0.8$-confidence upper bound, you may only get a $0.6$-confidence upper bound. 
Hence Figure \ref{fig:alphaPlots_ttest} shows the behavior that we would expect of a bound that does \emph{not} have guaranteed coverage, but which provides tight ``high-confidence'' upper bounds.

\begin{figure}[htbp]
    \centering
    \includegraphics[width=0.45\textwidth]{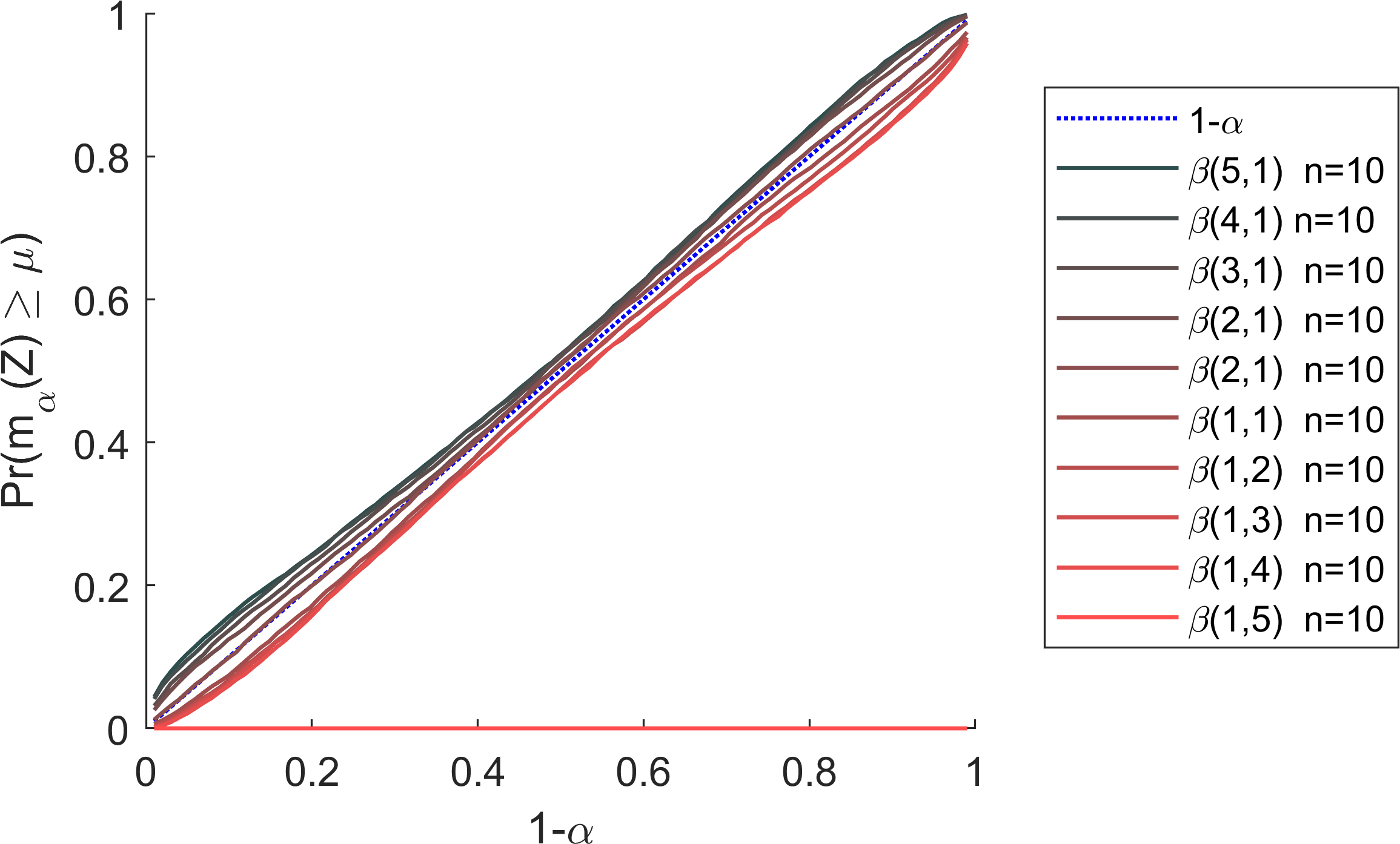}~
    \includegraphics[width=0.45\textwidth]{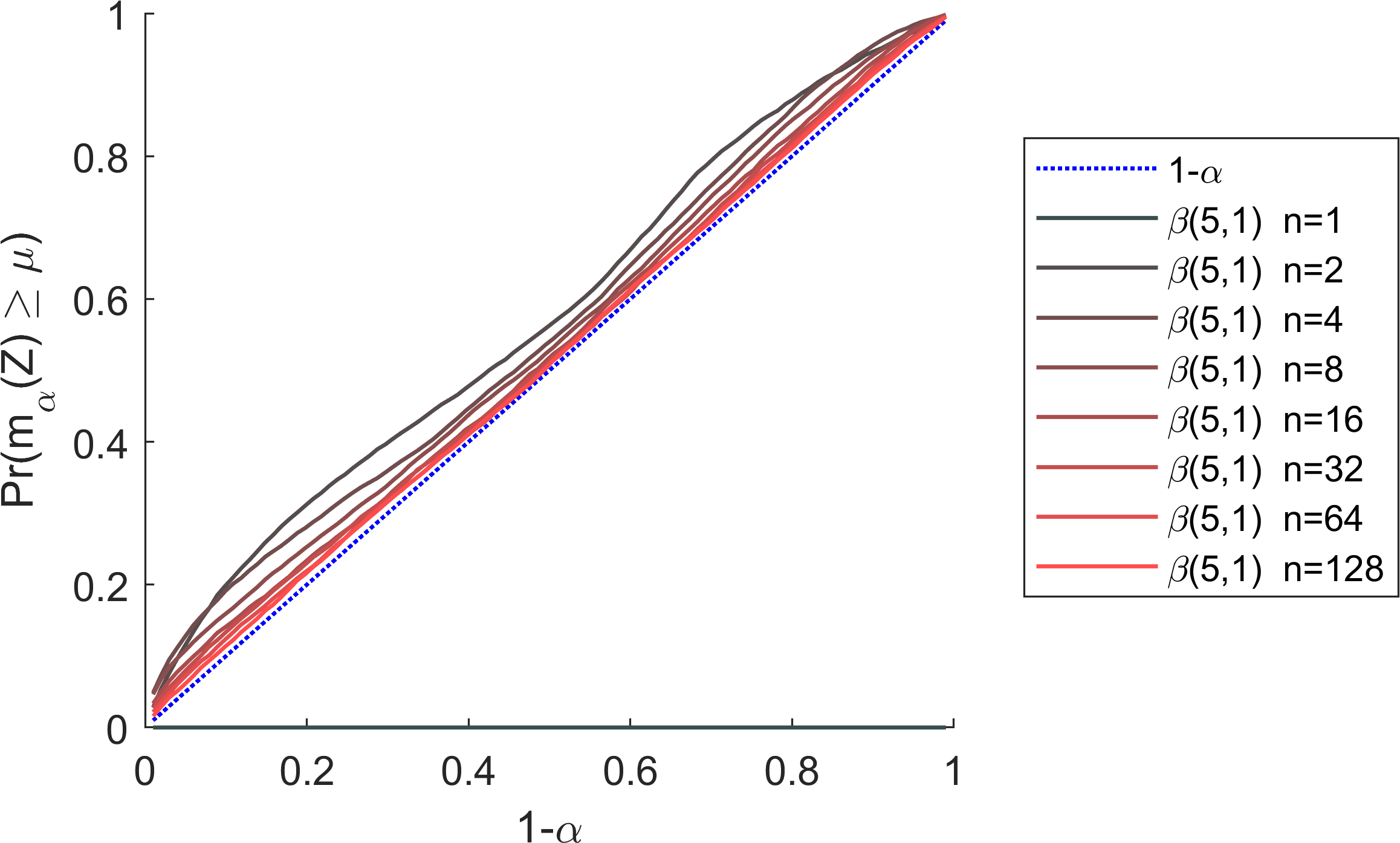}\\
    \includegraphics[width=0.45\textwidth]{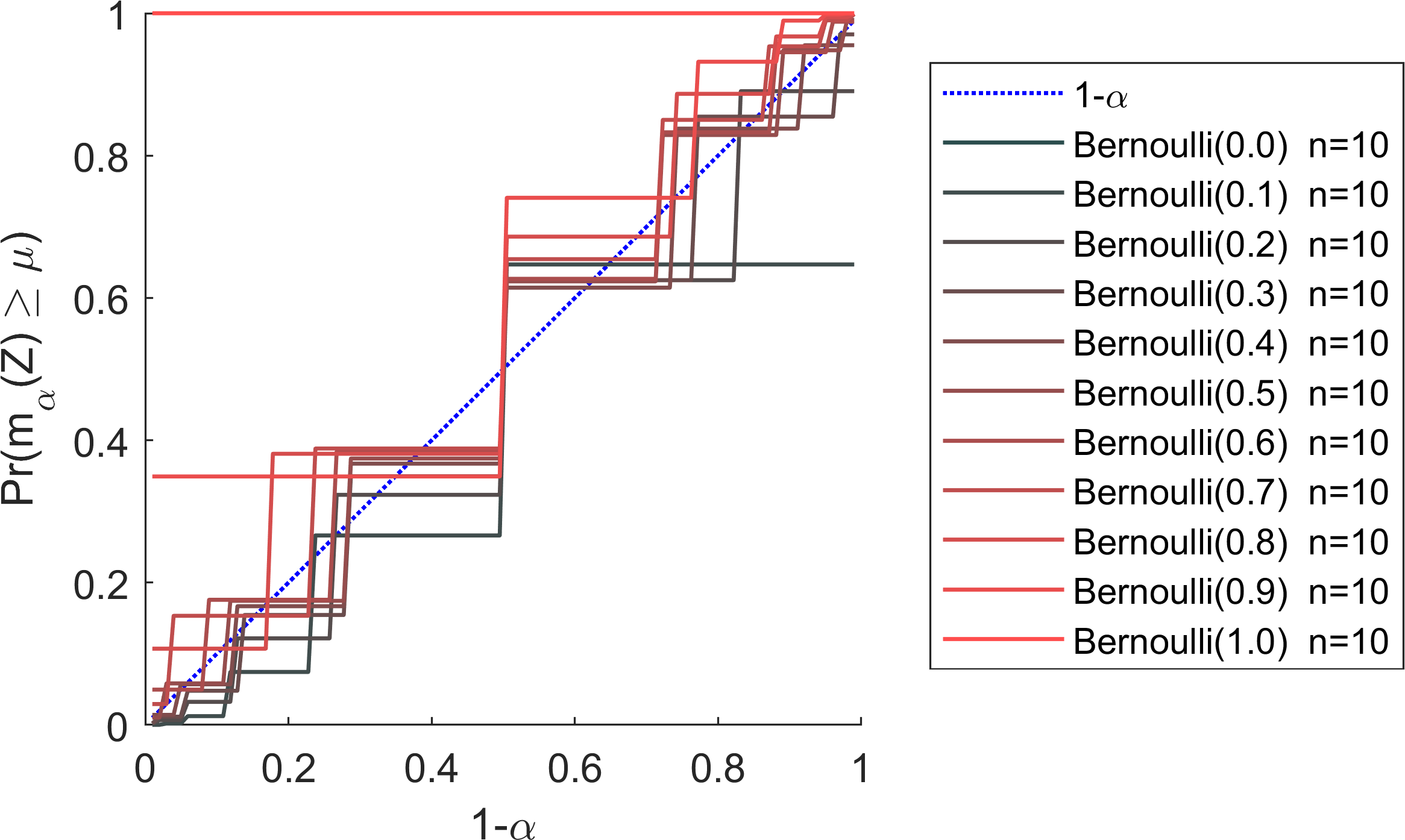}~
    \includegraphics[width=0.45\textwidth]{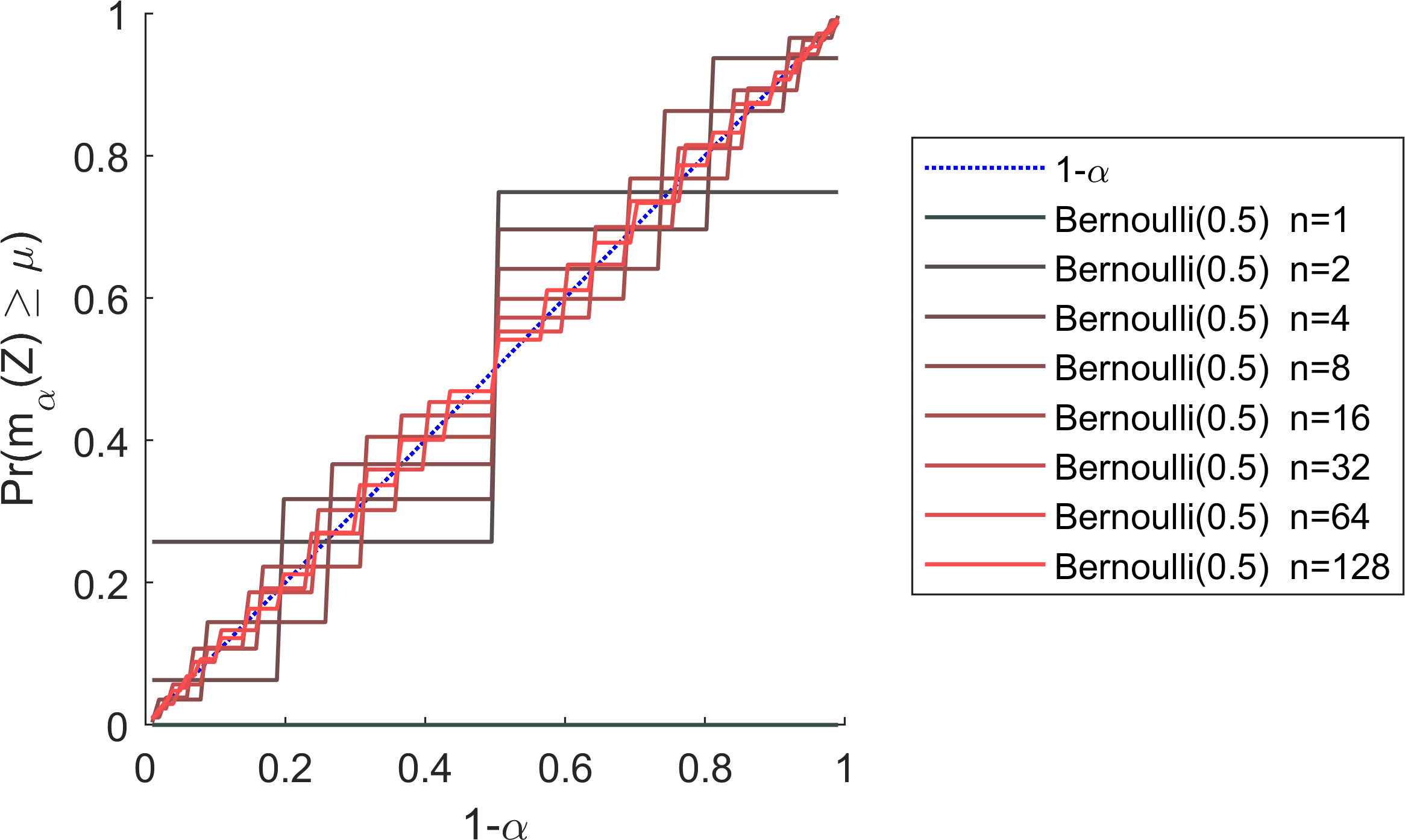}
    \caption{Estimated probability that the high-confidence upper bound produced using the Student-$t$ interval is greater than or equal to the true mean.}
    \label{fig:alphaPlots_ttest}
\end{figure}

The desired behavior of a high-confidence upper bound would blend the desirable properties of Hoeffding's inequality and the Student-$t$ interval. 
In these plots, this would result in curves that \emph{always} remain above the blue curve (guaranteed coverage), but are otherwise as low as possible (tight). 
Figure \ref{fig:alphaPlots} presents the results of this same experiment, conducted using our bound. 
It achieves this desired behavior---it always remains above the blue line (consistent with guaranteed coverage), but tends to be significantly lower than Hoeffding's inequality. 

\begin{figure}[htbp]
    \centering
    \includegraphics[width=0.45\textwidth]{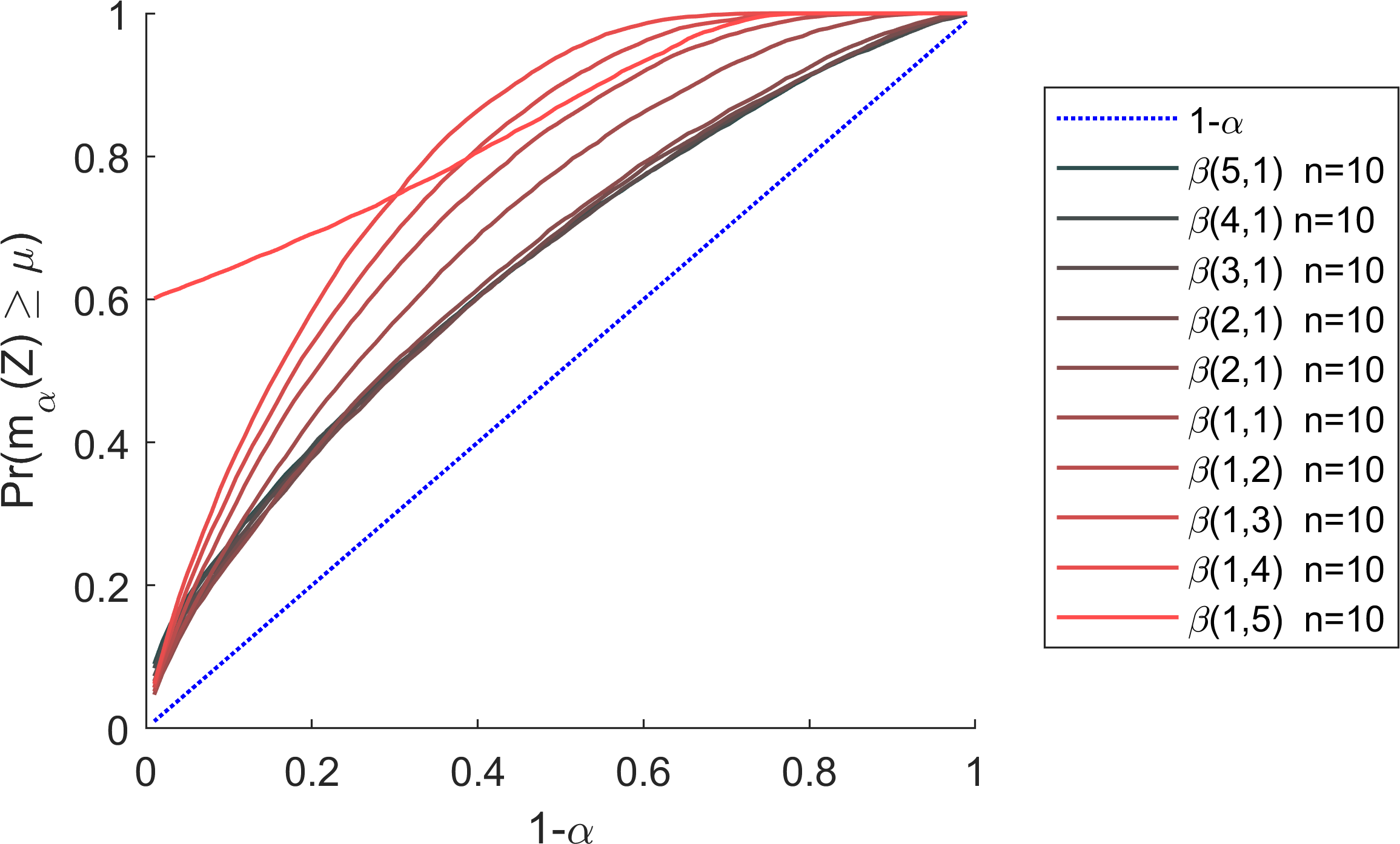}~
    \includegraphics[width=0.45\textwidth]{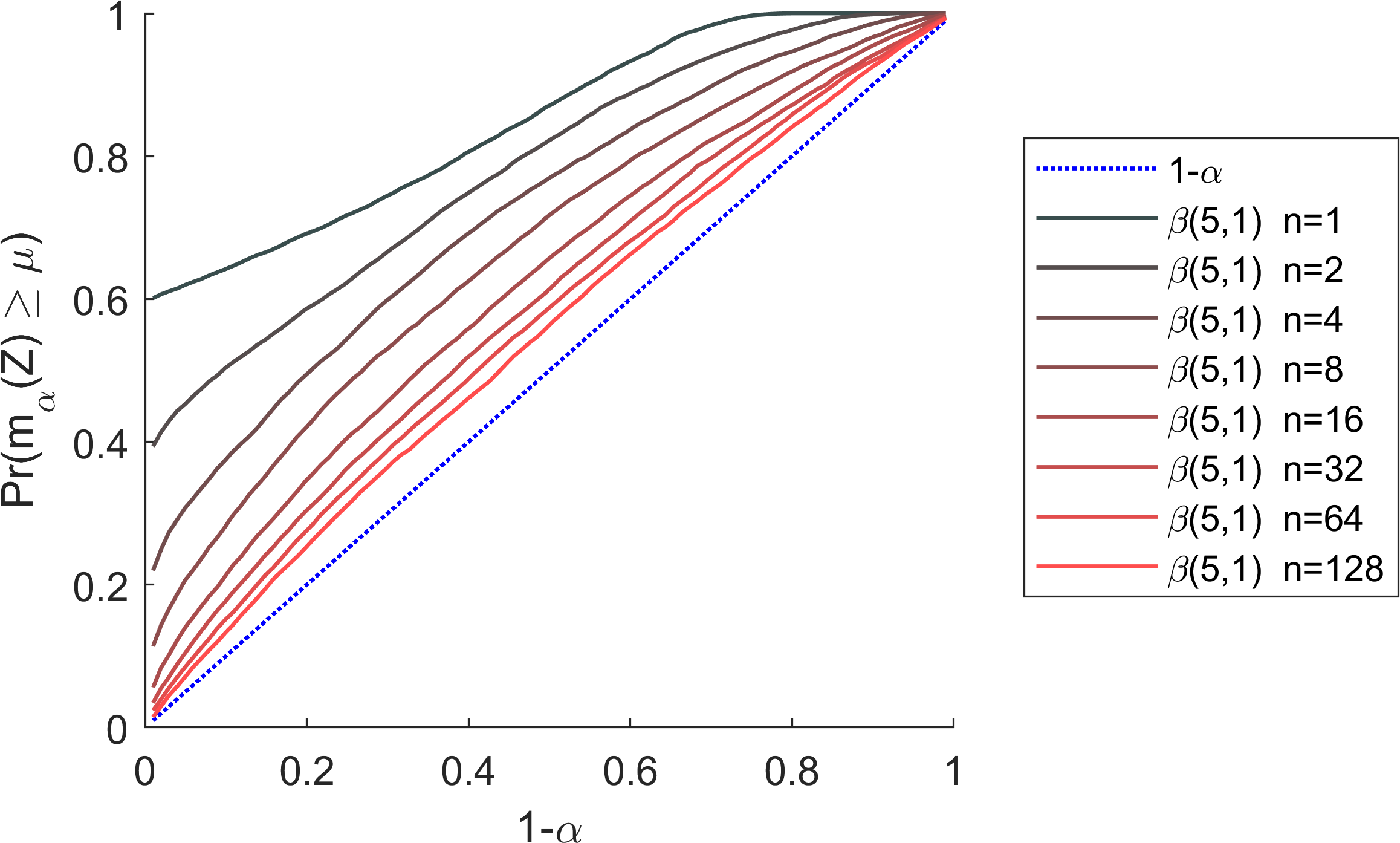}\\
    \includegraphics[width=0.45\textwidth]{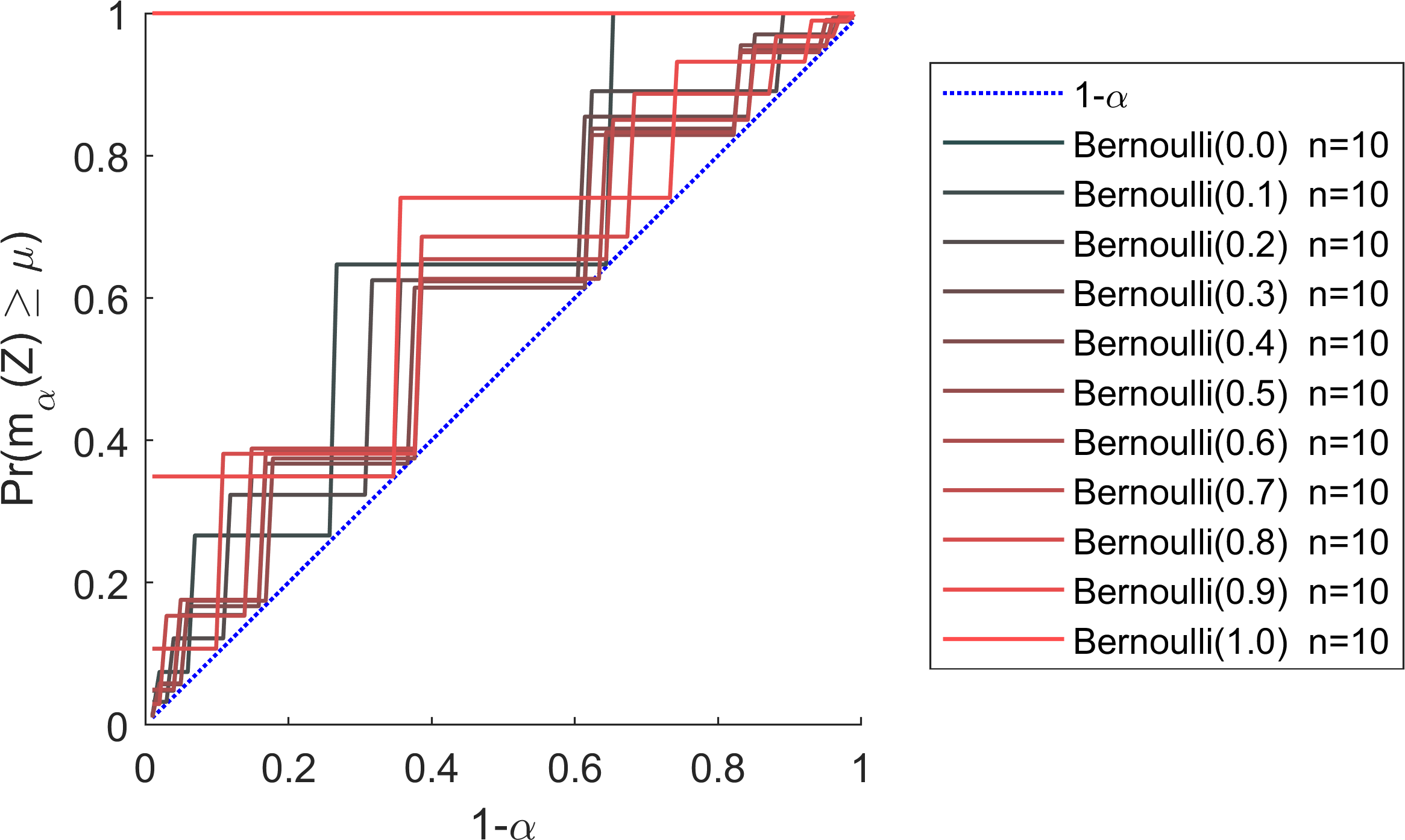}~
    \includegraphics[width=0.45\textwidth]{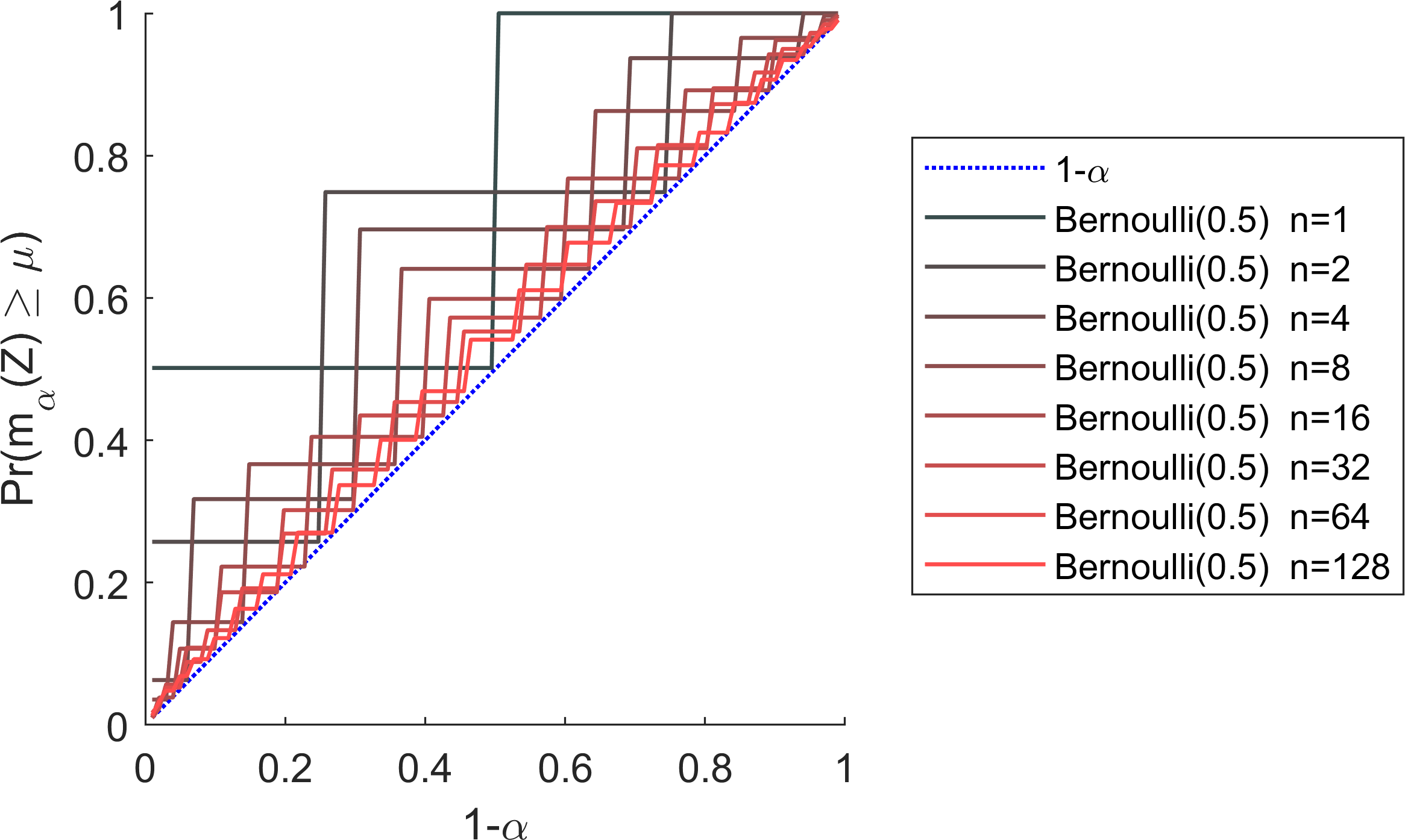}
    \caption{Estimated probability that the high-confidence upper bound produced using our bound is greater than or equal to the true mean.}
    \label{fig:alphaPlots}
\end{figure}

\subsection{Numerical Comparison to Previous Methods}

In this subsection we focus on RQ2 and RQ3 with experiments that compare the tightness of our bound to that of previous methods (both with and without guaranteed coverage). 
A variety of different statistics can be used to capture how tight the high-confidence upper bounds produced by a method are, including the mean upper bound, i.e., $\mathbf{E}[m_\alpha(\Z)]$, and the median upper bound. 
Here we report the mean upper bound: we gather $1,\!000$ samples of $\Z$ from a distribution and compute the upper bounds produced by our bound and several previous methods and report the sample mean of the upper bounds for each method. 
For simplicity, here we vary the distribution and $\N$ but fix $\alpha=0.05$ to obtain $95\%$-confidence upper bounds. 

\begin{figure}[htbp]
    \centering
    \includegraphics[width=0.45\textwidth]{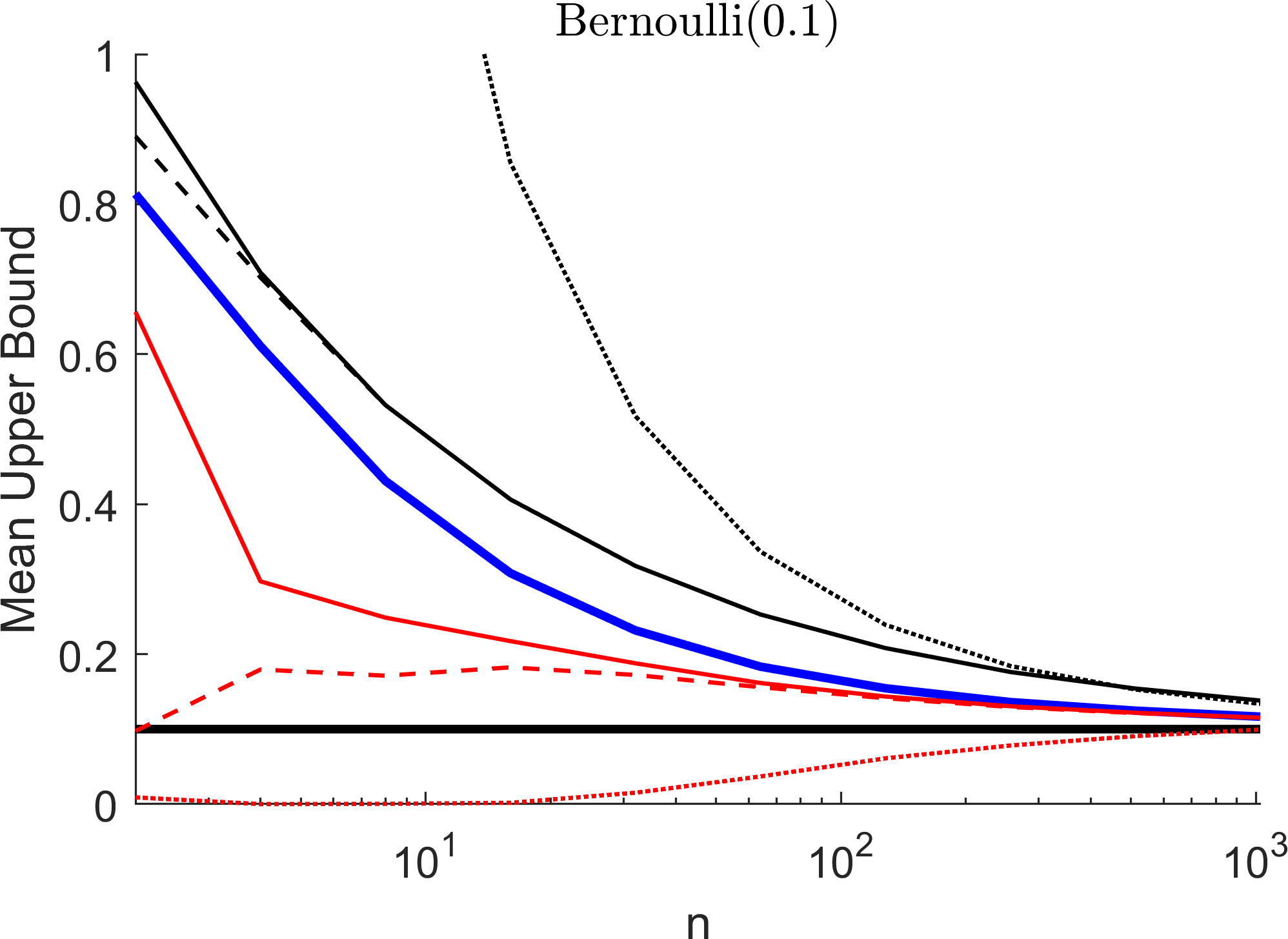}~
    \includegraphics[width=0.45\textwidth]{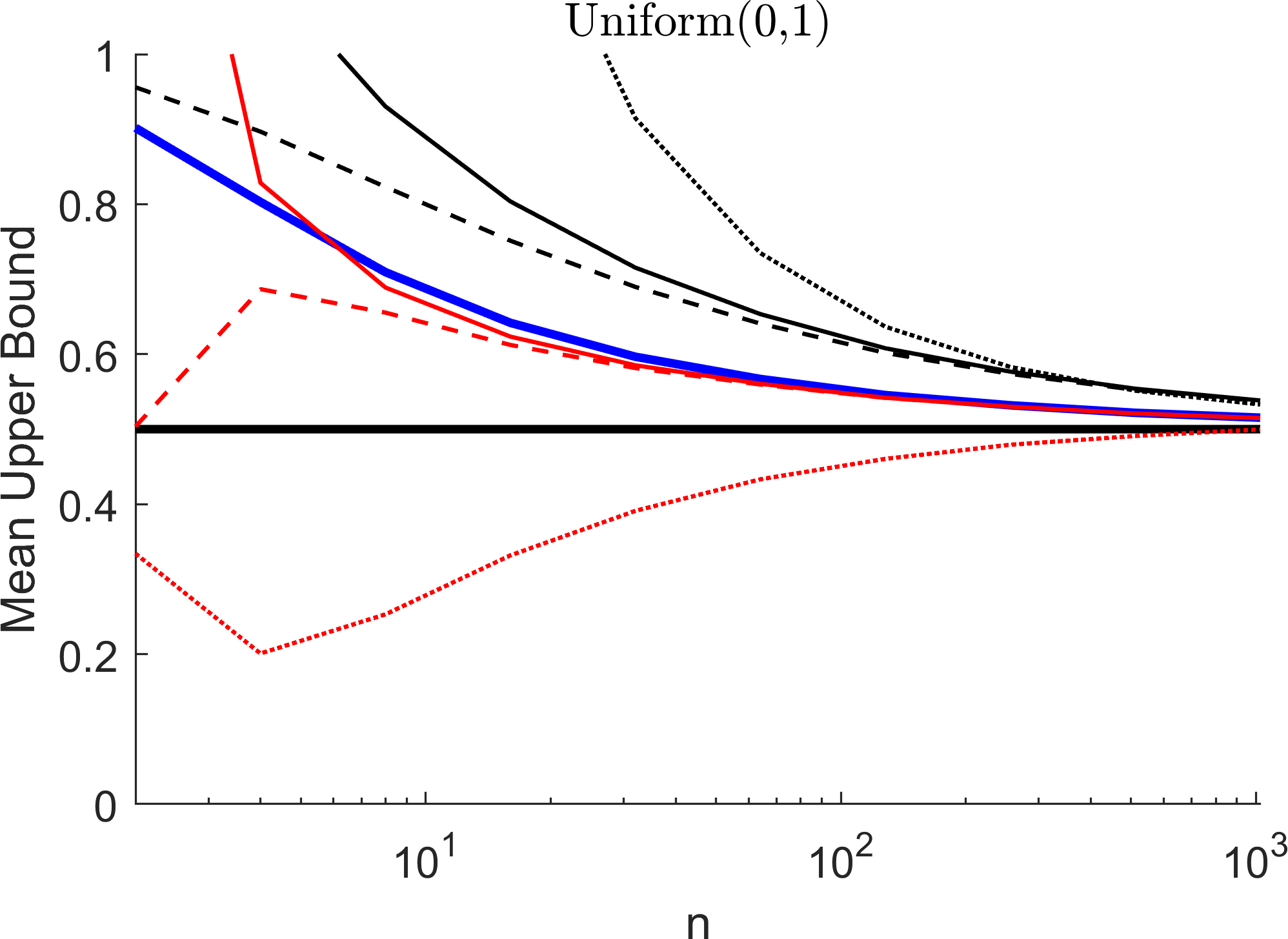}\\\vspace{1cm}
    \includegraphics[width=0.45\textwidth]{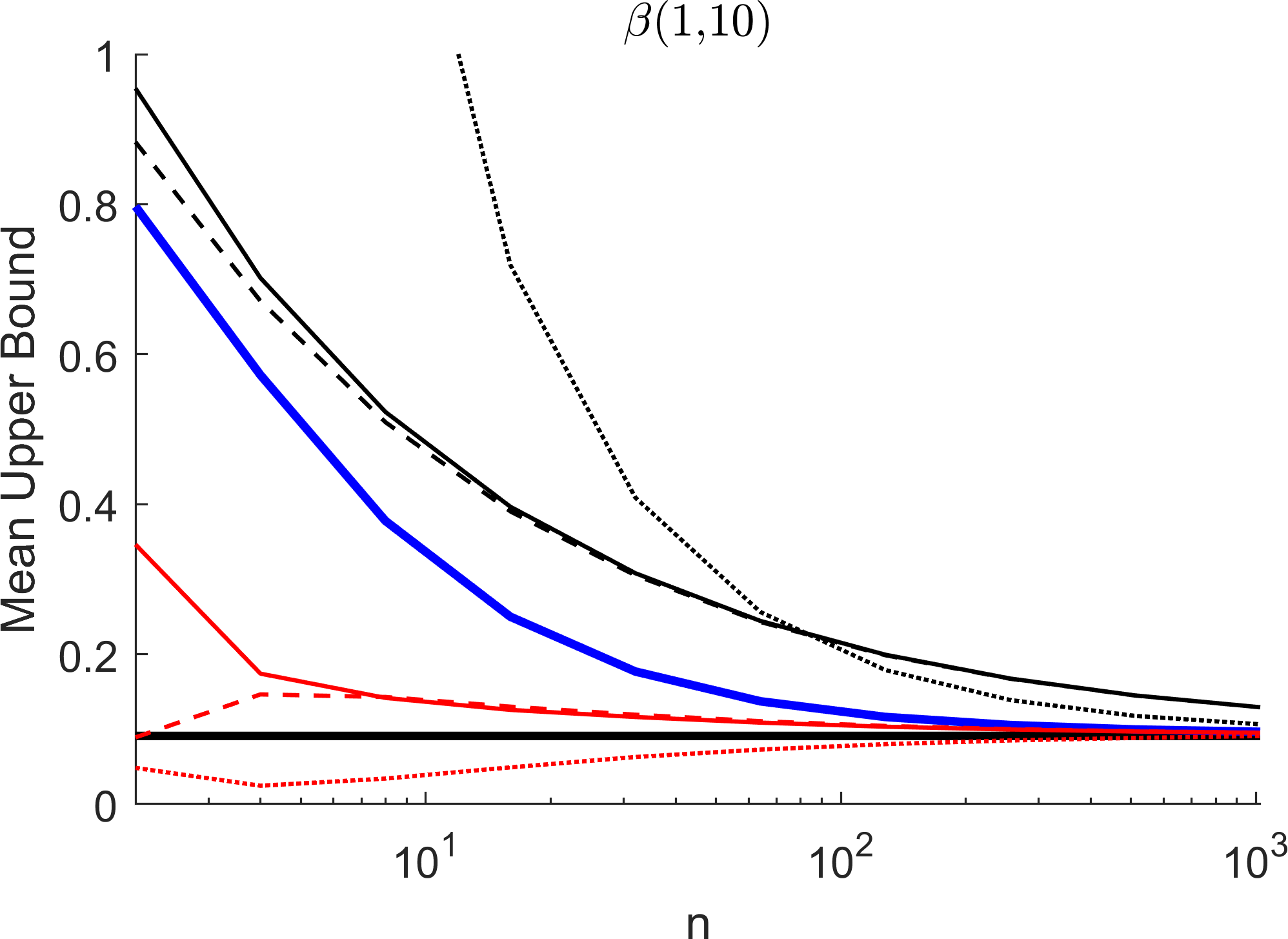}~
    \includegraphics[width=0.45\textwidth]{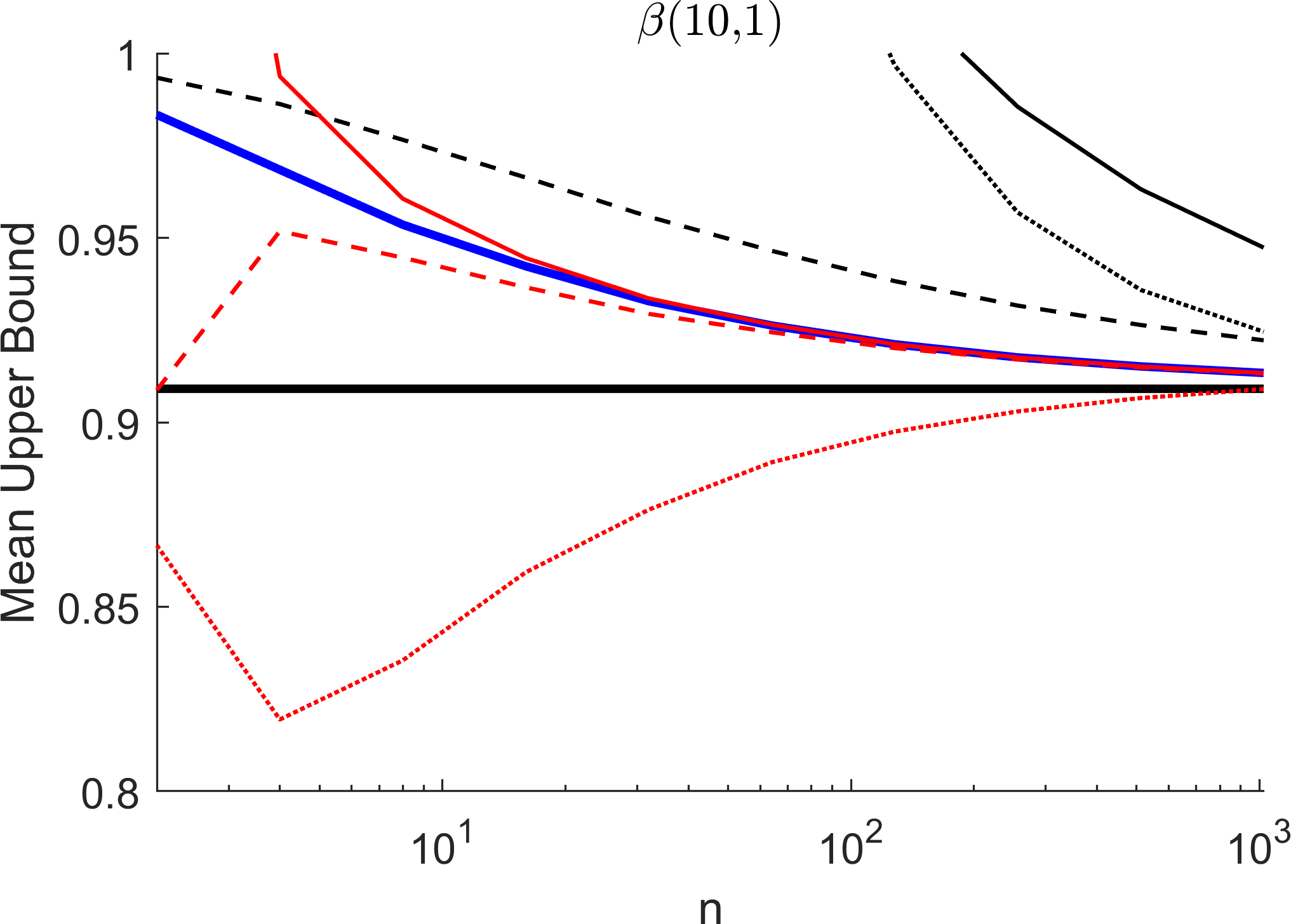} 
    \caption[These plots depict the mean upper bounds (over $1,\!000$ trials for various distributions (the titles on the plots describe the distribution) and using various methods. All figures share the following legend: Rendering Error]{These plots depict the mean upper bounds (over $1,\!000$ trials) for various distributions (the titles on the plots describe the distribution) and using various methods. All figures share the following legend: 
    \includegraphics[width=\textwidth]{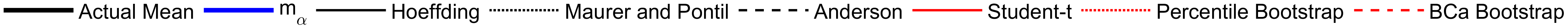}
    }
    \label{fig:nPlots}
\end{figure}

First compare the blue curve (our bound) to the black curves (previous methods with guaranteed coverage), noting that the horizontal axis uses a logarithmic scale. 
In every case, the blue curve remains strictly below the black curves, indicating that in every setting our bound produces lower values on average. 
Notice that frequently our bound obtains mean upper bounds that previous methods require an order of magnitude more samples to achieve, indicating that our bound is a drastic improvement in tightness and/or data efficiency.

Next, compare the blue curve (our bound) to the red curves (previous methods that do \emph{not} have guaranteed coverage). 
The two bootstrap methods do not provide guaranteed coverage, even with normality assumptions. 
So, although they produce tight confidence intervals (as is evident in these plots), the high-confidence bounds that they produce cannot be relied upon.

Next consider the Student-$t$ interval: for the uniform distribution, it produces high-confidence upper bounds that are similar to those produced by our bound. 
When the Student-$t$ interval is computed from normally distributed data, it produces a tighter high-confidence upper-bound than our bound. 
However, when the sampling distribution includes right-skew (e.g., $\beta(1,10)$), the Student-$t$ interval tends to be overly optimistic---it does not have guaranteed coverage (as is evident in Figure \ref{fig:alphaPlots_ttest} with $\beta(1,5)$). 
Hence, although the upper bound of the Student-$t$ interval tends to be lower than those produced by our method for $\beta(1,10)$, it does \emph{not} have guaranteed coverage. 
On the other hand, when the sampling distribution includes left-skew (e.g., $\beta(10,1)$, the Student-$t$ interval is overly-conservative (like Hoeffding's inequality). 
Hence, in Figure \ref{fig:nPlots}, the $\beta(10,1)$ plot indicates that our bound is \emph{tighter} than the Student-$t$ interval. 
This is further evidence that our bound is combining the desirable properties of Hoeffding's inequality and the Student-$t$ interval: it roughly preserves the tightness of the Student-$t$ interval, except in the cases where the Student-$t$ interval is too tight to provide guaranteed coverage (in which case our bound is sufficiently looser to provide guaranteed coverage).

\subsection{Numerical Support for Practical use With $\N < 30$}

Of the many potential uses of our bound, one stands out: it provides a valid method for constructing confidence intervals for scientific studies with fewer than $30$ samples. 
Even though the Student-$t$ interval does \emph{not} have guaranteed coverage when the sampling distribution is not normal, the central limit theorem tells us that the sample mean tends towards a normal distribution as $\N$ increases. 
Hence, the Student-$t$ interval becomes reasonable when $\N$ is large.\footnote{Notice that even with arbitrarily large $\N$, the Student-$t$ interval may \emph{not} have guaranteed coverage, so here saying that the Student-$t$ interval is ``reasonable'' does \emph{not} mean that it has guaranteed coverage.} 
However, without knowing the sampling distribution, it is not clear how large $\N$ must be for the Student-$t$ interval to be reasonable. 
A common rule of thumb used in current scientific research is that $\N$ must be at least $30$. 

This raises the question: what should one do when fewer than $30$ samples are available? 
Our bound provides an answer (assuming our conjecture is true), as it provides confidence intervals of comparable tightness, but with guaranteed support for \emph{any} $\N$ and without any normality assumptions. The only requirement is
the ability to identify limits on the support of the distribution.
To answer RQ3, we present an experiment that shows how our bound can be used to obtain confidence intervals based on fewer than $30$ empirical measurements.

Specifically, we used data from the United States Census from the year 2000 to obtain an estimate of the distribution of people's ages, considering only people zero to 84 years old. 
We then consider the problem of obtaining a tight high-confidence upper bound on the mean age of people ages zero to 84 based on $\N < 30$ samples. 
The results of this experiment are presented in 
Figure \ref{fig:agePlot}, which is a similar form to Figure \ref{fig:nPlots} (but without the logarithmic horizontal axis).  
The key observations from this plot are: \textbf{1)} previous methods with guaranteed coverage are too loose to provide useful high-confidence bounds with so few samples, \textbf{2)} the Student-$t$ interval is sufficiently tight, but it cannot be applied responsibly with such a small $\N$, and \textbf{3)} our bound produces high-confidence bounds that are comparable to the Student-$t$ interval (while maintaining guaranteed coverage, if our conjecture holds). 

\begin{figure}
    \centering
    \includegraphics[width=0.7\textwidth]{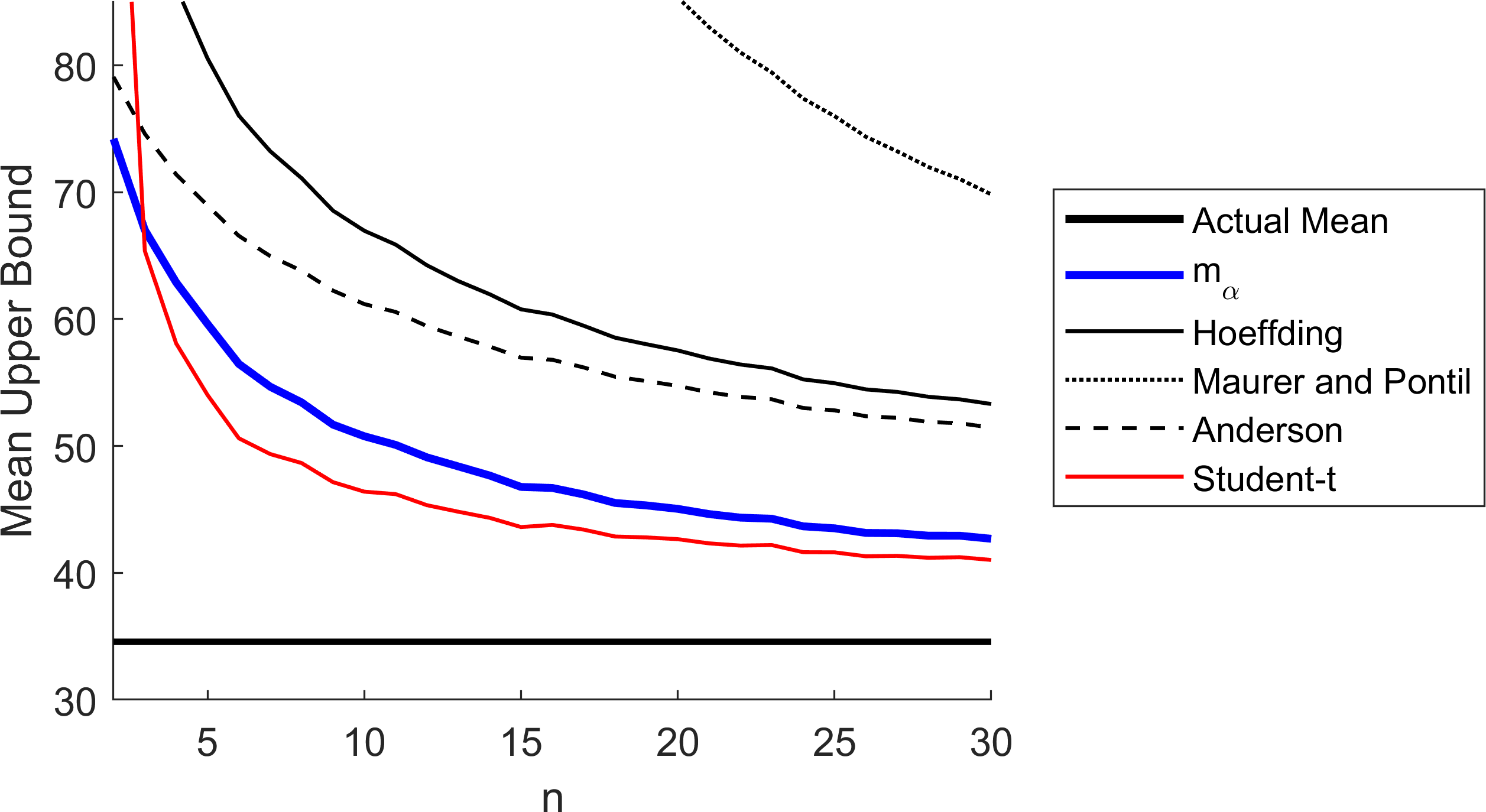}
    \caption{The mean upper bounds (over $1,\!000$ trials) produced by various methods when $\N$ is varied and the sampling distribution is an approximation of the distribution of ages (bounded in $[0,84]$) in the United States in the year 2000.}
    \label{fig:agePlot}
\end{figure}

\section{Acknowledgements}
This work benefitted significantly from conversations with Vince Lysinski and George Bissias, as well as from discussions with Andrew McGregor, Don Towsley, Berthold Horn, Archan Ray, Justin Domke, Gary Huang, Dan Sheldon, Luc Rey-Bellet, and Markos Katsoulakis. Some of this work grew out of early efforts to improve Anderson's bound by Benjamin Mears while at the University of Massachusetts.

\bibliographystyle{abbrvnat}

\end{document}